\pdfoutput=1
\documentclass[10pt,a4paper]{amsart}
\usepackage{amsmath, amssymb, amsfonts}
\usepackage[a4paper]{geometry}

\numberwithin{equation}{section}

\usepackage{stmaryrd}
\newcommand{\llb}{\llbracket}
\newcommand{\rrb}{\rrbracket}

\usepackage[scr=dutchcal]{mathalfa}

\renewcommand{\epsilon}{\varepsilon}
\renewcommand{\phi}{\varphi}
\renewcommand{\theta}{\vartheta}

\DeclareMathOperator{\Aut}{Aut}

\DeclareMathOperator{\diag}{diag}
\DeclareMathOperator{\End}{End}
\DeclareMathOperator{\Frac}{Frac}
\DeclareMathOperator{\Gal}{Gal}
\DeclareMathOperator{\GL}{GL}
\DeclareMathOperator{\Hom}{Hom}
\DeclareMathOperator{\id}{id}
\DeclareMathOperator{\ind}{ind}
\let\inf\relax \DeclareMathOperator{\inf}{inf}
\DeclareMathOperator{\Irr}{Irr}
\DeclareMathOperator{\nr}{nr}
\DeclareMathOperator{\rad}{rad}
\let\Re\relax\DeclareMathOperator{\Re}{Re}
\DeclareMathOperator{\res}{res}
\DeclareMathOperator{\Spec}{Spec}
\DeclareMathOperator{\Quot}{Quot}
\newcommand{\SK}{S\!K}
\DeclareMathOperator{\Tr}{Tr}

\newcommand{\AAA}{\mathfrak A}
\newcommand{\mm}{\mathfrak m}

\newcommand{\G}{\mathcal G}
\newcommand{\OO}{\mathcal O}
\newcommand{\Q}{\mathcal Q}

\newcommand{\CC}{\mathbb C}

\newcommand{\QQ}{\mathbb Q}

\newcommand{\ZZ}{\mathbb Z}

\newcommand{\al}{\mathrm{c}}
\newcommand{\cyc}{\mathrm{cyc}}
\newcommand{\eet}{\mathrm{\acute et}}
\newcommand{\ab}{\mathrm{ab}}
\newcommand{\cent}{\mathfrak z}

\usepackage{bbm}
\newcommand{\one}{\mathbbm 1}

\usepackage{tikz-cd}
\usepackage{colonequals}

\newcommand{\uplightning}{\rotatebox[origin=c]{180}{\ensuremath{\lightning}}}

\newenvironment{psmallmatrix}
{\left(\begin{smallmatrix}}
	{\end{smallmatrix}\right)}

\usepackage{tikz}
\usetikzlibrary{cd}
\usetikzlibrary{arrows}
\usetikzlibrary{positioning}

\usepackage[style=alphabetic,backend=biber,isbn=false]{biblatex}
\usepackage{enumitem}
\newlist{theoremlist}{enumerate}{1}
\setlist[theoremlist]{label=(\roman{theoremlisti}), ref=\thetheorem.\roman{theoremlisti},noitemsep, topsep=.2ex}
\newlist{propositionlist}{enumerate}{1}
\setlist[propositionlist]{label={(\roman{propositionlisti})}, ref=\theproposition.\roman{propositionlisti},noitemsep, topsep=.2ex}

\newcommand{\Bigcomplex}[6][1]{
	\begin{tikzcd}[row sep=-1ex, column sep=2em, ampersand replacement = \&, baseline=#1pt]
		\Big[{#2} \arrow[r, "#4"] \& {#5}\Big] \\
		{\scriptstyle #3} \& {\scriptstyle #6}
	\end{tikzcd}
}
\newcommand{\Hugecompl}[3][1]{
	\begin{tikzcd}[row sep=-1ex, column sep=2em, ampersand replacement = \&, baseline=#1pt]
		\Bigg[{#2} \Bigg] \\
		{\scriptstyle #3}
	\end{tikzcd}
}
\newcommand{\Hugecomplex}[6][1]{
	\begin{tikzcd}[row sep=-1ex, column sep=2em, ampersand replacement = \&, baseline=#1pt]
		\Bigg[{#2} \arrow[r, "#4"] \& {#5}\Bigg] \\
		{\scriptstyle #3} \& {\scriptstyle #6}
	\end{tikzcd}
}

\newcommand{\xrightarrowdbl}[2][]{%
	\xrightarrow[#1]{#2}\mathrel{\mkern-14mu}\rightarrow
}

\newcommand{\Dchi}{\mathfrak D}
\newcommand{\Z}{T}
\newcommand{\rcp}{\mathrm{rcp}}
\newcommand{\cp}{\mathrm{cp}}
\newcommand{\CCCb}{\mathbf C^{\mathrm b}}
\newcommand{\CCCbt}{\mathbf C^{\mathrm b}_{\mathrm{tor}}}
\newcommand{\Sigmachi}{\mathfrak O}
\DeclareMathAlphabet\EuScript{U}{eus}{m}{n}
\newcommand{\RWL}{\mathbb{L}}
\newcommand{\DCpt}{\DC^{\mathrm{perf}}_{\mathrm{tor}}}
\newcommand{\ordred}{\mathrm{ord}^\mathrm{red}}
\newcommand{\DC}{\mathbf D}
\newcommand{\LL}{\mathcal L}

\newcommand{\PPP}{\mathbf P}

\usepackage{wrapfig}
\usepackage{etoolbox} 

\usepackage{mathtools}
\let\underbrace\relax
\let\underbrace\LaTeXunderbrace

\usepackage[hidelinks,plainpages=false,pdfpagelabels]{hyperref}
\usepackage[capitalise, noabbrev]{cleveref}

\theoremstyle{plain}
\newtheorem{theorem}{Theorem}[section]
\newtheorem{lemma}[theorem]{Lemma}

\newtheorem{proposition}[theorem]{Proposition}
\newtheorem{corollary}[theorem]{Corollary}
\newtheorem{conjecture}[theorem]{Conjecture}

\newtheorem*{theorem*}{Theorem}
\newtheorem*{conjecture*}{Conjecture}
\newtheorem*{corollary*}{Corollary}
\theoremstyle{definition}
\newtheorem{definition}[theorem]{Definition}
\theoremstyle{remark}
\newtheorem{remark}[theorem]{Remark}

\Crefname{lemma}{Lemma}{Lemmata}
\Crefname{claim}{Claim}{Claims}
\Crefname{proposition}{Proposition}{Propositions}
\Crefname{conjecture}{Conjecture}{Conjectures}
\Crefname{example}{Example}{Examples}
\crefname{page}{page}{pages}
\Crefname{condition}{Condition}{Conditions}
\Crefname{question}{Question}{Questions}
\Crefname{theoremlisti}{Theorem}{Theorems}
\Crefname{propositionlisti}{Proposition}{Propositions}

\title{An equivariant $p$-adic Artin conjecture}
\author{Ben Forrás}
\address{Universität der Bundeswehr\\
	INF 1 Institut für Theoretische Informatik, Mathematik und Operations Research \\
	Werner-Heisenberg-Weg 39\\
	85579 Neubiberg\\
	Germany}
\email{ben.forras@unibw.de}
\urladdr{https://bforras.eu}

\date{Version of September 29, 2025}
\begin{document}

\begin{abstract}
We formulate an equivariant version of Greenberg's $p$-adic Artin conjecture for smoothed equivariant $p$-adic Artin $L$-functions in the context of an arbitrary one-dimensional admissible $p$-adic Lie extension of a totally real number field.
Using results of the author on the Wedderburn decomposition of the total ring of quotients of the Iwasawa algebra $\Lambda(\G)$, we deduce validity of the conjecture in several interesting cases.
\end{abstract}

\maketitle

\section*{Introduction}
Integrality questions have accompanied the study of $p$-adic Artin $L$-functions since its inception. Indeed, in the seminal works of Pi.~Cassou-Nogu\`es \cite{Cassou1979} resp. Deligne and Ribet  \cite{Ribet1979,DeligneRibet}, it was shown that in the case of a linear totally even $p$-adic Artin character $\chi$ with open kernel, the associated $p$-adic $L$-function can be expressed as a power series $G_{\chi}(X)$ divided by an explicit polynomial. These results were generalised to arbitrary $p$-adic Artin characters with open kernel by Greenberg \cite{Greenberg1983}, employing the method of Brauer induction. This led to a weakening of the integrality property above, with $G_{\chi}(X)$ now being a fraction of two power series. Greenberg posited two conjectures, namely the $p$-adic Artin conjecture and a refinement thereof, which assert that $G_{\chi}(X)$ is contained in the ring of power series over a certain ring of integers, possibly up to a constant factor. 
These conjectures are now theorems: {Greenberg showed that the $p$-adic Artin conjecture follows from the main conjecture, which was} proven by Wiles \cite{Wiles}, and its refinement was proven by Ritter and Weiss \cite{TEIT-II}. In this work, we propose a $p$-adic Artin conjecture for \emph{equivariant} $p$-adic Artin $L$-functions, and prove it in certain important cases.

\begin{wrapfigure}{o}{0.25\textwidth}
	\vspace*{-1cm}
	\begin{center}
		\begin{tikzcd}[every label/.append style = {font = \small}, every arrow/.append style={no head}, row sep=1cm]
			& L_\infty^+ \arrow[d] \arrow[ldd, "\G"] \\
			K_\infty \arrow[d, "\ZZ_p\simeq\Gamma"'] \arrow[ru, "H"] & L^+ \arrow[ld, no head] \\
			|[alias=K]|K \arrow[d] & \\[-.5cm]
			\QQ                                         
		\end{tikzcd}
	\end{center}
	\vspace*{-1cm}
\end{wrapfigure}

Let $\G$ be the Galois group of an admissible one-dimensional $p$-adic Lie extension $L^+_\infty/K$, where $L^+/K$ is the maximal totally real subextension of a finite CM extension $L/K$ such that $L$ contains a primitive $p$th root of unity, and $L^+_\infty$ is the cyclotomic $\ZZ_p$-extension of $L^+$. Let $S$ be a finite set of places of $K$ containing all places ramifying in $L_\infty^+/K$ and all infinite places. The group $\G$ is isomorphic to the semidirect product $H\rtimes\Gamma$  where $H\colonequals\Gal(L^+_\infty/K_\infty)$ is a finite group and $\Gamma\simeq \ZZ_p$. The results cited above are valid for any $p$-adic Artin character $\chi:\G\to\CC_p$ with open kernel, and the associated $p$-adic $L$-function $L_{p,S}(\chi,s)$ can be expressed as a fraction of a power series $G_{\chi,S}(X)$ and an explicit polynomial $H_\chi(X)$.

Let us now proceed to the equivariant setting, meaning that one treats all characters $\chi$ {of $\G$} simultaneously. The relevant object is the equivariant $p$-adic Artin $L$-function $\Phi_S\in \cent(\Q(\G))^\times$ constructed by Ritter--Weiss, where $\Q(\G)$ denotes the total ring of quotients of the Iwasawa algebra $\Lambda(\G)=\ZZ_p\llb\G\rrb$, and $\cent(\Q(\G))$ is the centre thereof. For each $p$-adic Artin character $\chi$ with open kernel, there is a specialisation map sending $\Phi_S$ to $G_{\chi,S}(\gamma-1)/H_\chi(\gamma-1)$, with $\gamma$ being a fixed topological generator of $\Gamma$. The explicit polynomial $H_\chi$ may be seen as a ``shadow'' of the trivial character, and shows that denominators will occur in $\Phi_S$, so one cannot expect an integrality property like the $p$-adic Artin conjecture for $\Phi_S$. A possible way to remedy this is the introduction of a nonempty smoothing set $T$ disjoint from $S$, and adjusting the characterwise $p$-adic $L$-functions by modified Euler factors at primes in $T$. This idea goes back to Tate \cite{Tate}, who used it in his treatise on the Brumer--Stark conjecture, now a theorem of Dasgupta--Kakde--Silliman--Wang \cite{BrumerStarkZ}. Smoothing at the set $T$, we obtain a smoothed equivariant $p$-adic Artin $L$-function $\Phi^T_S\in \cent(\Q(\G))^\times$. Then in analogy with Greenberg's $p$-adic Artin conjecture, we posit the following integrality statement:
\begin{conjecture*}[equivariant $p$-adic Artin conjecture]
	Let $\Gamma_0\colonequals\Gamma^{p^{n_0}}$ where $n_0$ is large enough such that $\Gamma_0$ is central in $\G$.
	Let $\mathfrak M$ be a $\Lambda(\Gamma_0)$-order in $\Q(\G)$ containing $\Lambda(\G)$. Then the smoothed equivariant $p$-adic Artin $L$-function $\Phi_S^T$ is in the image of the composite map
	\begin{equation*}
		\mathfrak M\cap \Q(\G)^\times \to K_1(\Q(\G)) \xrightarrow{\nr} \cent(\Q(\G))^\times,
	\end{equation*}
	where the first arrow is the natural map sending an invertible element to the class of the $1\times 1$ matrix consisting of said element, and the second arrow is the reduced norm map.
\end{conjecture*}

\subsection*{Main result}
In this paper, we prove the following maximal order variant of the equivariant $p$-adic Artin conjecture:
\begin{theorem*}
	Assume the equivariant Iwasawa main conjecture (\cref{EIMCu}). Then the equivariant $p$-adic Artin conjecture holds for all maximal $\Lambda(\Gamma_0)$-orders in $\Q(\G)$ containing $\Lambda(\G)$.
\end{theorem*}

The proof of this statement rests upon three pillars: (i) a dimension reduction result, (ii) a description of the Wedderburn decomposition of $\Q(\G)$, and (iii) a smoothed main conjecture. We briefly discuss these ingredients here.

The dimension reduction result goes back to Nichifor and Palvannan \cite{NichiforPalvannan}. They studied integrality of non-smoothed equivariant $p$-adic Artin $L$-functions attached to extensions cut out by mod $p$ characters, with the assumption that $\G$ is a direct product of $H$ and $\Gamma$. A key step in the proof is a dimension reduction statement involving Dieudonné determinants of matrices and employing Venjakob's noncommutative Weierstraß theory \cite{VenjakobWPT}. Our dimension reduction result is an adaptation of their method to the semidirect product case.

Addressing the equivariant $p$-adic Artin conjecture, in particular the reduced norm map, necessitates studying the Wedderburn decomposition of the semisimple artinian ring $\Q(\G)$. If one assumes $\G$ to be a direct product, the representation theory of $\Q(\G)$ is rather straightforward: in this case, the Wedderburn decomposition of $\Q(\G)$ is directly determined by that of $\QQ_p[H]$, and the maximal order in Nichifor--Palvannan's theorem is induced by the unique maximal orders in each of the skew fields in the Wedderburn decomposition.
However, when $\G$ is a semidirect product, the difference between Wedderburn decompositions of $\QQ_p[H]$ and $\Q(\G)$ becomes substantial: for instance, when $\G$ is pro-$p$, all Schur indices for $\QQ_p[H]$ are $1$ due to a theorem of Schilling \cite[(74.15)]{CRII}, whereas $\Q(\G)$ can have nontrivial Schur indices. A general description of the Wedderburn decomposition of $\Q(\G)$ has been carried out in \cite{W}.

In recent work of Johnston and Nickel \cite{NABS}, an $(S,T)$-modified Iwasawa module $Y_S^T$ was introduced. The module $Y_S^T$ admits a free resolution of length one, and one may formulate a smoothed main conjecture relating the class of $Y_S^T$ in $K_0(\Lambda(\G),\Q(\G))$ to a $T$-smoothed version $\Phi_S^T$ of Ritter--Weiss's equivariant $p$-adic Artin $L$-function, postulating the existence of a unique element $\zeta_S^T\in K_1(\Q(\G))$ such that $\nr(\zeta_S^T)=\Phi_S^T$ and $\zeta_S^T$ is mapped to the class of $Y_S^T$ under the connecting homomorphism of relative $K$-theory; see \cref{sEIMCu}. Under the assumption of the main conjecture, the statement of the equivariant $p$-adic Artin conjecture can also be reformulated in terms of $\zeta_S^T$ (\cref{integrality-zeta}).
The algebraic object in the main conjecture poses an important difference between the present work and that of Nichifor--Palvannan: indeed, the Iwasawa module they consider is the Pontryagin dual of a certain non-smoothed Selmer group, and a cohomological criterion due to Greenberg shows that this module only admits a free resolution of length $1$ away from the trivial character. No such restriction applies to the Johnston--Nickel module $Y_S^T$.

The equivariant Iwasawa main conjecture has been proven in the $\mu=0$ case by Ritter--Weiss \cite{RW-MC} and Kakde \cite{Kakde}, as well as in the case when $\G$ has an abelian $p$-Sylow subgroup by Johnston--Nickel \cite{UAEIMC}. In these cases, we obtain the following result:
\begin{corollary*}
	Suppose that either $L^+_\infty/K$ satisfies the $\mu=0$ hypothesis, or that $\G$ has an abelian $p$-Sylow subgroup. Then the equivariant $p$-adic Artin conjecture holds for every maximal $\Lambda(\Gamma_0)$-order in $\Q(\G)$.
\end{corollary*}

\subsection*{Special case: \texorpdfstring{$p$}{p}-abelian extensions}
The equivariant $p$-adic Artin conjecture is stated for all $\Lambda(\Gamma_0)$-orders $\mathfrak M$ in $\Q(\G)$ containing $\Lambda(\G)$. The smaller $\mathfrak M$ is, the stronger this statement becomes, the extreme being $\mathfrak M=\Lambda(\G)$. While proving this in general seems to be out of reach at the present juncture, it can be proven in the so-called $p$-abelian case, in which it is an easy corollary of a theorem of Johnston and Nickel \cite{UAEIMC}.

\begin{corollary*}
	If $p$ does not divide the order of the commutator subgroup of $\G$, then $\Phi_S^T\in\cent(\Lambda( \G))$. Moreover, {the equivariant $p$-adic Artin conjecture} holds with $\mathfrak M=\Lambda( \G)$.
\end{corollary*}

\subsection*{Outline}
This paper is structured as follows. \Cref{sec:prelims} recalls the algebraic background used in the sequel. \Cref{sec:smoothed-eq-p-adic-Artin-L} is of an analytic nature: we recall the definitions of characterwise and equivariant $p$-adic Artin $L$-functions, and define a smoothed version. We formulate the equivariant $p$-adic Artin conjecture in \cref{sec:conjectures}. In \cref{sec:MC}, we recall the relevant algebraic object, and state a smoothed equivariant main conjecture. We use this in \cref{sec:generalities-on-integrality} to study the behaviour of the equivariant $p$-adic Artin conjecture under change of the input data. \Cref{sec:dimension-reduction} is a generalisation of Nichifor--Palvannan's dimension reduction result. Finally in \cref{sec:known-results}, we prove the equivariant $p$-adic Artin conjecture in the cases mentioned above.

The paper contains results obtained in Chapters 1 and 4 of the author's doctoral thesis \cite{thesis}.

\subsection*{Funding} The results in the aforementioned thesis were obtained while the author was a member of the Research Training Group 2553 at the Essen Seminar for Algebraic Geometry and Arithmetic (ESAGA), funded by the Deutsche Forschungsgemeinschaft (DFG, project no.~412744520).

\subsection*{Acknowledgements} The author wishes to extend his gratitude to Bharathwaj Palvannan for hosting him at IISc Bengaluru in November 2022, and for many helpful discussions. Furthermore, the author thanks Andreas Nickel for his guidance and support throughout his doctoral studies, as well as for his comments on a draft version of this article.

\subsection*{Notation and conventions} The letter $p$ always stands for an odd prime number. We fix an algebraic closure $\QQ_p^\al$ of $\QQ_p$.

The word ``ring'' is short for ``not necessarily commutative ring with unity''. A domain is a ring with no zero divisors; the term ``integral domain'' is reserved for commutative domains. A principal ideal domain (PID) is a domain such that all left ideals and all right ideals are principal; in particular PIDs are not necessarily commutative.  If $R$ is a ring, $\Quot(R)$ stands for the total ring of quotients of $R$, which is obtained from $R$ by inverting all central regular elements; when $R$ is an integral domain, this is the field of fractions, and to emphasise this, we use the notation $\Frac(R)$ instead.

For a ring $R$, the ring of $n\times n$ matrices is $M_n(R)$. The $n\times n$ identity matrix is $\mathbf 1_n$, {whereas the trivial character is $\one$}.
The centre of a group $G$ resp. a ring $R$ is denoted by $\cent(G)$ resp. $\cent(R)$.
If $S$ is a finite set, then $\# S$ denotes its cardinality.
For $n\ge1$, $\mu_n$ stands for the group of $n$th roots of unity. For a field $F$, we write $\mu(F)$ for the group of roots of unity in $F$.

In our notation, we make a clear distinction between rings of power series and completed group algebras: we use double square brackets for the former and blackboard square brackets for the latter. E.g. $\ZZ_p[[X]]$ is a ring of power series, and $\ZZ_p\llb \Gamma\rrb$ is a completed group algebra. If $\G$ is a profinite group, then we write $\Irr(\G)$ for the set of $\QQ_p^\al$-valued irreducible characters of $\G$ with open kernel.

We will abuse notation by writing $\oplus$ for a direct product of rings, even though this is not a coproduct in the category of rings.

\section{Algebraic preliminaries} \label{sec:prelims}
This section consists of a recollection of algebraic concepts used in subsequent sections.
\subsection{Iwasawa algebras, characters and idempotents} \label{sec:Iwasawa-prelim}
Let $\G$ be a profinite group isomorphic to $H\rtimes\Gamma$, where $H$ is a finite group and $\Gamma\simeq\ZZ_p$. The Iwasawa algebra of $\G$ over $\ZZ_p$ is
$\Lambda(\G)\colonequals \ZZ_p\llb \G\rrb$.
Let $\Q(\G)\colonequals\Quot(\ZZ_p\llb\G\rrb)$ denote its total ring of quotients. We will also consider $\Q^\al(\G)\colonequals \QQ_p^\al\otimes_{\QQ_p}\Q(\G)$. These algebras are defined analogously for $\G$ an arbitrary profinite group. The ring $\Q(\G)$ is semisimple artinian, as shown in the proof of \cite[Proposition~5(1)]{TEIT-II}. On the other hand, the Iwasawa algebra $\Lambda(\G)$ is never semisimple.

Let $\Gamma_0\colonequals\Gamma^{p^{n_0}}$ where $n_0$ is chosen such that $\Gamma_0\subseteq\G$ is central. This is the case when $n_0$ is large enough: indeed, since $\G$ is a semidirect product, failure of $\Gamma$ to be central comes from the homomorphism $\phi:\Gamma\to\Aut(H)$ defining conjugation by elements of $\Gamma$ not being trivial. However, since $H$ is finite, so is $\Aut(H)$, wherefore $\phi$ has open kernel in $\Gamma$. All such open subgroups are of the form $\Gamma^{p^n}$ for some $n\ge0$. Therefore if $\ker(\phi)=\Gamma^{p^n}$ then $\Gamma^{p^{n_0}}$ is central in $\G$ whenever $n_0\ge n$. Fix a topological generator $\gamma$ of $\Gamma$, and let $\gamma_0\colonequals\gamma^{p^{n_0}}$.

As $\Q(\G)$ is semisimple, it can be studied further by investigating its Wedderburn decomposition, which leads to considering characters of $\G$ and $H$. Let $\Irr(\G)$ denote the set of absolutely irreducible $\QQ_p^\al$-valued characters of $\G$ with open kernel, and let $\chi\in\Irr(\G)$. Let $\eta\mid\res^\G_H\chi$ be an irreducible constituent of the restriction of $\chi$ to $H$.

Following \cite[\S1]{NickelConductor}, let us define the fields
\begin{align}
	\QQ_{p,\chi} &\colonequals \QQ_p(\chi(h): h\in H) \label{eq:Qpchi};\\
	\QQ_p(\eta) &\colonequals \QQ_p(\eta(h): h\in H) \label{eq:Qpeta}.
\end{align}
Write $w_\chi\colonequals(\G:\G_\eta)$ for the index of the stabiliser $\G_\eta$ of $\eta$ in $\G$. This is a power of $p$ since $H$ stabilises $\eta$. Moreover, neither $\QQ_p(\eta)$ nor $w_\chi$ depend on the choice of the irreducible constituent $\eta$.
The field $\QQ_p(\eta)$ is an extension of $\QQ_{p,\chi}$.
Furthermore, the field $\QQ_p(\eta)$, and thus also the subfield $\QQ_{p,\chi}$, is abelian over $\QQ_p$, as it is contained in some cyclotomic extension of $\QQ_p$. For $g\in \G$ and $\sigma\in\Gal(\QQ_p(\eta)/\QQ_p)$, the characters ${}^g\eta$ resp. ${}^\sigma\eta$ of $H$ are given by {${}^g\eta(h)=\eta(g^{-1}hg)$} resp. ${}^\sigma\eta(h)=\sigma(\eta(h))$.

\subsection{Skew power series rings} \label{sec:skew-power-series-rings}
Let $R$ be a noetherian pseudocompact ring, that is, it is a complete Hausdorff noetherian topological ring with a fundamental system $(\mathscr L_i)_{i\in I}$ of open neighbourhoods of zero such that $\mathscr L_i\subseteq  R$ is a left ideal and each $ R/\mathscr L_i$ is a finite length $ R$-module. 
Let $\sigma\in\End R$ be a ring endomorphism, and let $\delta: R\to  R$ be a $\sigma$-derivation, that is, a homomorphism of additive groups such that for all $r,s\in R$, $\delta(rs)=\delta(r)s+\sigma(r)\delta(s)$.
Suppose that $\delta$ is $\sigma$-nilpotent. We refrain from giving the definition of $\sigma$-nilpotency in general. In the case when $\sigma$ and $\delta$ commute, which is the only case we shall encounter in this work, this is equivalent to $\delta$ being topologically nilpotent, that is, for all $n\ge1$ there {being} an $m\ge1$ such that for all $k\ge m$, there is an inclusion $\delta^k(R)\subseteq (\rad  R)^n$, where $\rad R$ is the Jacobson radical of $ R$.

For such a ring $R$, endomorphism $\sigma$, and derivation $\delta$, one can define the formal skew power series ring $ R[[X;\sigma,\delta]]$ as the ring with underlying additive group $ R[[X]]$ and multiplication defined by $Xr=\sigma(r)X+\delta(r)$ for $r\in R$. Nilpotence of $\delta$ is needed to ensure that multiplication is well-defined, that is, the coefficients of any product converge in $ R$. We refer to \cite[\S\S0--1]{SchneiderVenjakob} for details. In the sequel, we drop the adjective ``formal''.

Let $R$ be a local ring with maximal ideal $\mm=\rad R$ such that $ R$ is separated and complete with respect to the $\mm$-adic topology. Then skew power series over $ R$ possess a Weierstraß theory, as exhibited in \cite[\S3]{VenjakobWPT}. For a skew power series $f(X)=\sum_{i=0}^\infty a_i X^i\in  R[[X;\sigma,\delta]]$, let 
\begin{equation*}\label{eq:ordred}
	\ordred(f)\colonequals\inf\{i:a_i\in R^\times\}\in \ZZ_{\ge0}\cup\{\infty\}
\end{equation*}
be the lowest degree term with an invertible coefficient, called the reduced order of $f$. Then the following hold:

\begin{theorem}[Weierstraß division theorem] \label{WDT}
	If $f\in R[[X;\sigma,\delta]]$ has finite reduced order, then
	\[ R[[X;\sigma,\delta]]= R[[X;\sigma,\delta]] f\oplus\bigoplus_{i=0}^{\ordred f-1}  R[[X;\sigma,\delta]] X^i.\]
\end{theorem}

\begin{theorem}[Weierstraß preparation theorem] \label{WPT}
	If $f\in R[[X;\sigma,\delta]]$ has finite reduced order, then there is a unique unit $\epsilon\in R[[X;\sigma,\delta]]^\times$ and a unique distinguished skew polynomial $F\in  R[X;\sigma,\delta]$ such that $f=\epsilon F$.
\end{theorem}
Here a monic skew polynomial being distinguished means that all coefficients but the leading one are contained in the maximal ideal $\mm$.

\subsection{Algebraic \texorpdfstring{$K$}{K}-theory}
We recall the relevant $K$-groups used in equivariant Iwasawa theory, mostly following the exposition in \cite[\S4.1]{hybrid}.
Let $ R$ be a noetherian integral domain, $ A$ a finite dimensional algebra over $\Frac( R)$, and $\AAA$ an $ R$-order in $ A$. Later we will set these to be $ R\colonequals\Lambda(\Gamma_0)$, $ A\colonequals\Q(\G)$ and $\AAA\colonequals\Lambda(\G)$.

The groups $K_0(\AAA)$ resp. $K_0( A)$ are the Grothendieck groups of the abelian categories $\PPP(\AAA)$ resp. $\PPP( A)$ of finitely generated projective left $\AAA$- resp. $ A$-modules; see \cite[\S1]{Sujatha} for details. We let $K_1(\AAA)$ resp. $K_1( A)$ denote the Whitehead groups of the rings $\AAA$ resp. $ A$, as defined e.g. in {\cite[\S40]{CRII}}.

Let $\CCCb(\PPP(\AAA))$  be the category of bounded complexes of finitely generated projective $\AAA$-modules, and let $\CCCbt(\PPP(\AAA))$ be the full subcategory of such complexes with $ R$-torsion cohomology groups. We define the relative $K_0$-group $K_0(\AAA, A)$ to be the abelian group generated by the objects of $\CCCbt(\PPP(\AAA))$, with the relations being $[C^\bullet]=0$ for $C^\bullet$ acyclic, and $[C^\bullet]=[C'^\bullet]+[C''^\bullet]$ for short exact sequences $0\to C'^\bullet\to C^\bullet\to C''^\bullet\to0$ in $\CCCbt(\PPP(\AAA))$. In fact, $\CCCbt(\PPP(\AAA))$ is a Waldhausen category, and $K_0(\AAA, A)=K_0(\CCCbt(\PPP(\AAA)))$; see \cite[II, Definition~9.1.2 and p.~185]{Kbook}.

An equivalent definition of the relative $K_0$-group is to let $K_0(\AAA, A)$ be the abelian group generated by triples $(P,\phi,Q)$ where $P$ and $Q$ are finitely generated projective $\AAA$-modules and $\phi: A\otimes_\AAA P\xrightarrow\sim A\otimes_\AAA Q$ is an isomorphism of $ A$-modules. The relations are
\begin{align*}
	(P,\phi,Q)+(Q,\psi,R)&=(P,\psi\phi,R), \\
	(P',\phi',Q')+(P'',\phi'',Q'')&=(P,\phi,Q),
\end{align*}
where $P'\hookrightarrow P\twoheadrightarrow P''$ and $Q'\hookrightarrow Q\twoheadrightarrow Q''$ are exact sequences of $\AAA$-modules, and the diagram
\[
\begin{tikzcd}
	 A\otimes_\AAA P' \arrow[d, "\sim", "\phi'"'] \arrow[r] &  A\otimes_\AAA P \arrow[d, "\sim", "\phi"'] \arrow[r] &  A\otimes_\AAA P'' \arrow[d, "\sim", "\phi''"'] \\
	 A\otimes_\AAA Q' \arrow[r]                   &  A\otimes_\AAA Q \arrow[r]                   &  A\otimes_\AAA Q''                  
\end{tikzcd}
\]
is commutative. See \cite[II, Definition~2.10]{Kbook} for more on this. An isomorphism between this group and the one defined above is given by
\begin{align} \label{eq:K0-rel-K0-iso}
	[(P,\phi,Q)] &\mapsto \Bigcomplex{P}{-2}{\phi}{Q}{-1}
\end{align}
on triples $(P,\phi,Q)$ for which $\phi$ maps $P$ to $Q$.
Here the complex on the right is concentrated in degrees $-2$ and $-1$.
For further details and other definitions, we refer the reader to \cite[\S2]{Sujatha}.

The aforementioned $K$-groups are connected by the localisation exact sequence of $K$-theory {\cite[(40.9)]{CRII}}:
\begin{equation}\label{eq:localisation-K}
	K_1(\AAA)\to K_1( A)\xrightarrow{\partial}K_0(\AAA, A)\to K_0(\AAA)\to K_0( A).
\end{equation}
The first resp. last arrows are induced by the inclusion $\AAA\to A$ using functoriality of Whitehead resp. Grothendieck groups. The connecting homomorphism is given as follows: if $M\in \GL_n( A)$, and $\widetilde M\in \End(\AAA^n)$ such that $1_ A\otimes \widetilde M=M$, then
\begin{equation} \label{eq:def-partial}
	\partial([M])\colonequals\Bigcomplex{\AAA^n}{-2}{\widetilde M}{\AAA^n}{-1},
\end{equation}
where the complex is in degrees $-2$ and $-1$. Matrices $M$ for which such an $\widetilde{M}$ exists generate $K_1(A)$: {indeed, for any  $X\in\GL_n(A)$, there is a $0\ne d\in R$ such that $dX\in M_n(\mathfrak A)\cap\GL_n(A)$ by clearing denominators; then $[X]=[dX]\cdot [d^{-1}\mathbf 1_n]$, with $dX\in M_n(\mathfrak A)$ and $d^{-1}\mathbf 1_n$ being the inverse of an element in $M_n(\mathfrak A)$.}
Finally, the third arrow {in \eqref{eq:localisation-K}} sends a complex $[P\xrightarrow{\phi} Q]$ concentrated in degrees $-2$ and $-1$ to $[P]-[Q]$.

In the case $ R\colonequals\Lambda(\Gamma_0)$, $ A\colonequals\Q(\G)$ and $\AAA\colonequals\Lambda(\G)$, the connecting homomorphism $\partial$ is surjective by \cite[Corollary~3.8]{Witte}; see \cite[\S4.1]{hybrid} for further comments on this. Therefore the localisation sequence takes the following form:
\begin{equation} \label{Witte}
	K_1(\Lambda(\G))\to K_1(\Q(\G))\xrightarrowdbl\partial K_0(\Lambda(\G),\Q(\G))\hookrightarrow K_0(\Lambda(\G))\to K_0(\Q(\G)).
\end{equation}

Finally, we recall some derived categorical language, which will make the formulation of some statements easier. We let $\DC(\AAA)$ denote the derived category of $\AAA$-modules. A complex of $\AAA$-modules is perfect if, considered as an element of $\DC(\AAA)$, it is isomorphic to a complex in $\CCCb(\PPP(\AAA))$. The full subcategory of perfect complexes with $ R$-torsion cohomology modules will be denoted by $\DCpt(\AAA)$. With each object of $\DCpt(\AAA)$, one can naturally associate an element in $K_0(\AAA, A)$. In the notation $\DCpt(\AAA)$, the domain $ R$ is suppressed; in this work, we will always have $ R\colonequals\Lambda(\Gamma_0)$ and $\AAA\colonequals\Lambda(\G)$, so this should not lead to confusion.

\subsection{Maximal orders}
We will use the fact that maximal orders behave well with respect to taking matrix rings or direct sums:
\begin{proposition}[{\cite[Theorem~8.7]{MO}}] \label{MO8.7}
	Let $ R$ be a noetherian integrally closed integral domain with field of fractions $ F$. Then if $\Delta$ is a maximal $ R$-order in a separable $ F$-algebra $ A$, then the matrix ring $M_n(\Delta)$ is a maximal $ R$-order in $M_n( A)$ for all $n\ge1$.
\end{proposition}

\begin{proposition}[{\cite[Theorem~10.5(ii)]{MO}}]\label{MO10.5ii}
	Suppose that a separable $ F$-algebra $ A$ has decomposition
	\( A=\bigoplus_{i=1}^n  A_i\)
	into simple components. Then if for all $i=1,\ldots,n$, the ring $\Delta_i$ is a maximal $ R$-order in $ A_i$, then $\bigoplus_{i=1}^n \Delta_i$ is a maximal $ R$-order in $ A$.
\end{proposition}

\subsection{Skew fields over local fields} \label{sec:skew-fields-over-local-fields}
We begin with some generalities on cyclic algebras; see \cite[\S30]{MO} for details. Let $E/ F$ be a finite cyclic Galois extension of fields of degree $s$, with $\varsigma$ a generator of $\Gal(E/ F)$. Let $ a\in  F^\times$ be a unit. Then the cyclic algebra $ A\colonequals(E/ F,\varsigma, a)$ is defined to be the $ F$-algebra
\begin{equation*}
	(E/ F,\varsigma, a)\colonequals \bigoplus_{i=0}^{s-1} E \alpha^i,
\end{equation*}
where $\alpha$ is a formal $s$th root of $ a$, and $\alpha$ obeys the multiplication rule $\alpha x \alpha^{-1} =\varsigma(x)$ for all $x\in E$. 

The theory of skew fields over local fields was originally laid out by Hasse in \cite{Hasse}, a presentation of which is available in \cite[Chapter 3]{MO}. We provide a brief review of the results.

Let $ R$ be a complete commutative discrete valuation ring, $ F=\Frac R$ its field of fractions. Let $ D$ be a skew field with $\cent( D)= F$ and index $s$.
By extending the valuation from $ F$ to $ D$, one can show that $ D$ contains a unique maximal $ R$-order, which is the integral closure of $ R$ in $ D$.

Under the additional assumption that the residue field of $ R$ is finite, the skew field $ D$ can be described explicitly: see \cite[\S14]{MO}. In this case, $ F$ is a ($1$-dimensional) local field. Write $q$ for the order of the residue field of $ R$. Let $ F(\omega)$ be the cyclotomic field extension of $ F$ obtained by adjoining a primitive $(q^s-1)$st root of unity $\omega$. Then $ F(\omega)$ is a maximal subfield of $ D$, called the inertia field (unique up to conjugacy in $ D$).

Fix a uniformiser $\pi\in F$. Then there exists a uniformiser $\pi_{ D}$ in the maximal order such that $\pi_{ D}^s=\pi$, and there exists $1\le  r\le s$ coprime to $s$ such that $\pi_{ D}\omega\pi_{ D}^{-1}=\omega^{ q^ r}$. The fraction $ r/s\in\QQ/\ZZ$ is called the Hasse invariant of $ D$. Let $\sigma\in\Gal( F(\omega)/ F)$ be the automorphism defined by $\sigma(\omega)=\omega^{ q^ r}$; this is a power of the Frobenius and a generator of the Galois group $\Gal( F(\omega)/ F)$. The maximal $\ZZ_p$-order in $ D$ is generated by $\pi_{ D}$ and $\omega$ as an $ R$-algebra. In other words, $ D$ is the cyclic algebra
\begin{equation*}
	 D=\big(  F(\omega) /  F, \sigma, \pi_ F \big)=\bigoplus_{i=0}^{s-1}  F(\omega) \pi_{ D}^i.
\end{equation*}

Since $ F(\omega)$ is a maximal subfield of $ D$, it is also a splitting field. A splitting isomorphism can be described explicitly as follows.
\begin{align}
	 F(\omega)\otimes_ F D&\xrightarrow{\sim}M_s( F(\omega)) \label{eq:Reiner-splitting} \\
	x\otimes 1&\mapsto x\mathbf 1_s \notag \\
	1\otimes x&\mapsto \diag\left(x, \sigma(x),\ldots,\sigma^{s-1}(x)\right) \notag \\
	1\otimes \pi_ D &\mapsto \begin{psmallmatrix}
		0&1 \\ &&1 \\ \vdots&&&\ddots \\ &&&&1 \\ \pi &&& \dots & 0 \\
	\end{psmallmatrix} \notag
\end{align}
Here $x\in  F(\omega)$. We let $\phi$ denote the composite map
\begin{align}
	\phi:  D \to  F(\omega)\otimes_ F D&\xrightarrow{\sim}M_s( F(\omega)), \label{eq:phi-def}
\end{align}
where the first arrow is $x\mapsto 1\otimes x$ and the second one is \eqref{eq:Reiner-splitting}. Then $\phi$ maps $ D$ isomorphically onto its image (as rings).

\subsection{Reduced norms and the Dieudonné determinant} \label{sec:Nrd-det}
We recall definitions of the reduced norm map, the Dieudonné determinant, and some of their key properties. Standard references on reduced norms include \cite[§7D]{CR} and \cite[§9]{MO}. Caveat: while our definition of the reduced norm on a central simple algebra agrees with that of \cite[\S9a]{MO}, its extension to semisimple algebras differs from that in {\cite[§7D]{CR} and \cite[\S9b]{MO}}.
For the Dieudonné determinant, see \cite[III\S1]{Kbook} and \cite[165--166]{CR}. {An overview} with a view towards Iwasawa algebras is provided in \cite[pp.~616--617]{NichiforPalvannan}; we shall follow the exposition there.

Let $ F$ be a field. Let $ A$ be a central simple $ F$-algebra. Let $E/ F$ be a splitting field for $ A$, meaning that there is an isomorphism
\[\phi: E\otimes_ F A \xrightarrow{\sim}M_{n}( E)\]
of $ E$-algebras. A splitting field always exists by \cite[Proposition~7.25]{CR}.
The reduced characteristic polynomial of $a$ is defined to be
\[\rcp_{ A/ F}(a)\colonequals \cp(\phi(1\otimes a)),\]
where $\cp(\phi(1\otimes a))$ is the characteristic polynomial of the matrix $\phi(1\otimes a)$ over the field $ E$. The reduced characteristic polynomial does not depend on the choice of the splitting field $ E$, and its coefficients lie in $ F$.

The reduced norm, which can be seen as a noncommutative generalisation of the determinant, is defined as follows:
\[\nr_{ A/ F}(a)\colonequals(-1)^{\deg\rcp_{ A/ F}(a)}\rcp_{ A/ F}(a)(0).\]
For every $n\ge1$, the matrix ring $M_n( A)$ is also a central simple $ F$-algebra, and the above definition provides maps $\nr_{M_n( A)/ F}$. On invertible matrices, these take values in $ F^\times$, and they induce a reduced norm map on the colimit {$\GL(A)$}, which factors through the abelianisation {$K_1(A)$}, {thereby} defining a map $\nr: K_1( A)\to F^\times$. The kernel of this map is denoted by $\SK_1( A)\colonequals\ker(\nr)$.

The following result shows that elements in orders have integral reduced norms.
\begin{proposition}[{\cite[Theorem~10.1]{MO}}] \label{MO10.1}
	Let $ R$ be a noetherian integrally closed integral domain, $ F=\Frac( R)$ its field of fractions, $ A$ a central simple $ F$-algebra, and $\Delta$ an $ R$-order in $ A$. Then the reduced characteristic polynomial of an element in $\Delta$ has coefficients in $ R$.
\end{proposition}

The notion of a reduced norm extends to arbitrary semisimple $ F$-algebras as follows: if $ A=\bigoplus_{i=1}^n  A_i$, where each $ A_i$ is a central simple algebra with centre $ F_i$, then there is a reduced norm map $\nr_{ A_i/ F_i}$ in each component, and their product defines a homomorphism
\[\nr_{ A/\cent( A)}\colonequals\bigoplus_{i=1}^n \nr_{ A_i/ F_i}:  A\to \cent( A)=\bigoplus_{i=1}^n  F_i.\]

There is another generalisation of the usual determinant: the Dieudonné determinant. Let $ R$ be a semilocal ring (that is, $ R/\rad  R$ is semisimple artinian).
Define the subgroup $W( R)\le  R^\times$ as
\[W( R)\colonequals \{(1+rs)(1+sr)^{-1}: (1+rs)\in  R^\times\}.\]
Then the Dieudonné determinant $\det:\GL_n( R)\to  R^\times/W( R)$ is the homomorphism characterised by the following two properties: 
\begin{enumerate}
	\item Elementary matrices have Dieudonné determinant $1$. An elementary matrix is a matrix which differs from the identity matrix by a single off-diagonal entry.
	\item If $r\in R$ and $M\in \GL_n( R)$ then $\det(\diag(r,1,\ldots,1) M)=r \det(M)$.
\end{enumerate}
In particular, the Dieudonné determinant of a(n upper or lower) triangular matrix is the product of the elements in its diagonal. This makes sense since the Dieudonné determinant is only defined up to $W( R)$, which contains the commutator subgroup of $ R^\times$. (This is not obvious; consult either \cite[Theorem~3.7]{Brenner} or \cite[III, Lemma~1.3.3 and Exercise~1.1]{Kbook} for a proof.)

We indicate how to compute $\det A$ for $A\in M_n( R)$. 
We say that an $n\times n$ matrix is a permutation matrix if it is obtained from the identity matrix $\mathbf 1_n$ by swapping two of its columns. This is the terminology used in \cite{NichiforPalvannan}; some authors also call a product of such matrices a permutation matrix.
Suppose that we can write $A=UBV$ where $B$ is diagonal, and $U$ and $V$ are products of elementary, permutation and scalar matrices. In this case it is said that $A$ admits a diagonal reduction via elementary operations.
As explained above, we can compute the Dieudonné determinant of any diagonal matrix, and the same goes for elementary and scalar matrices. Permutation matrices have order at most $2$, hence their Dieudonné determinants are second roots of unity---in fact, the Dieudonné determinant of a permutation matrix is always $-1$, see \cite[pp.~35--36]{Dieudonne}.

Jacobson's elementary reduction theorem provides an important class of rings for which there is such  a decomposition $A=UBV$. We say that $ R$ is a PID if all left ideals and all right ideals are principal, see \cite[Chapter~3, \S1]{Jacobson}. Then the elementary reduction theorem states the following:
\begin{theorem}[{\cite[Chapter~3, Theorem~16]{Jacobson}}] \label{thm:Jacobson-reduction}
	If $ R$ is a PID, then every square matrix over $ R$ admits a diagonal reduction via elementary operations.
\end{theorem}

{The Dieudonné determinant is related to reduced norms as follows.}
If $ R= D$ is a skew field, then the Dieudonné determinant is a homomorphism $\det: \GL_n( D)\to ( D^\times)^\ab$, and the reduced norm factors over it, see \cite[(7.42)]{CR}:
\begin{equation} \label{eq:nr-det}
	\begin{tikzcd}
		\GL_n( D) \arrow[rr, "\nr_{M_n( D)/\cent( D)}"] \arrow[rd, "\det"'] & & \cent( D)^\times \\ & ( D^\times)^\ab \arrow[ru, "\nr_{ D/\cent( D)}"'] &         
	\end{tikzcd}
\end{equation}

{The Dieudonné determinant naturally defines} a homomorphism $K_1( R)\to  R^\times/W( R)$.
Due to results of Vaserstein, there is an isomorphism $ R^\times/W( R)\simeq ( R^\times)^\ab$ whenever $ R$ is a local ring or $2\in R^\times$, see \cite[III, Example~1.3.7]{Kbook}. Moreover, \cite[Theorem~2]{Vaserstein} states that in this case, the Dieudonné determinant provides an isomorphism
\begin{equation} \label{eq:Dieudonne-det-iso}
	K_1( R)\xrightarrow{\sim}( R^\times)^\ab.
\end{equation}
The ring $\Q(\G)$ is semisimple and therefore semilocal. The Iwasawa algebra $\Lambda(\G)$ is finitely generated as a module over the commutative local ring $\Lambda(\Gamma_0)$, hence it is semilocal, see \cite[Proposition~5.28(ii)]{CR}. Therefore the Dieudonné determinant is defined for both $\Q(\G)$ and $\Lambda(\G)$.

We will sometimes talk about reduced norms of elements in $\Lambda(\G)$. In this case, the map is understood to be the reduced norm map $\nr_{\Q(\G)/\cent(\Q(\G))}$, as the ring $\Lambda(\G)$ is not necessarily semisimple but it is contained in the semisimple ring $\Q(\G)$. In particular, the reduced norm of such an element need not lie in $\Lambda(\G)$, but see \cref{p-abelian-integrality} for a special case in which this is true.

\section{Smoothed equivariant \texorpdfstring{$p$}{p}-adic Artin \texorpdfstring{$L$}{L}-functions} \label{sec:smoothed-eq-p-adic-Artin-L}
In this section, we work on the analytic side. We first recall the definitions of characterwise resp. equivariant $p$-adic Artin $L$-functions, then introduce a smoothed version of the latter, and investigate its basic properties.

\subsection{Characterwise \texorpdfstring{$p$}{p}-adic Artin \texorpdfstring{$L$}{L}-functions}
Let $p$ an odd rational prime, and let $K$ be a number field.
Let $K_\infty$ be the cyclotomic $\ZZ_p$-extension of $K$. Let $\LL/K$ be a Galois extension such that $\LL$ is totally real, contains $K_\infty$, and such that the degree $(\LL:K_\infty)$ is finite. In other words, $\LL/K$ is an admissible one-dimensional $p$-adic Lie extension. The Galois group $\G\colonequals\Gal(\LL/K)$ is isomorphic to the semidirect product of the profinite group $\Gamma\colonequals\Gal(K_\infty/K)\simeq\ZZ_p$ and the finite group $H\colonequals\Gal(\LL/K_\infty)$, with $H$ being a normal subgroup: $\G\simeq H\rtimes\Gamma$. Finally, let $S$ be a finite set of places of $K$ containing all infinite places.

For all places $v\notin S$ of $K$, fix a place $w_\infty$ of $\LL$ above $v$. Let $I_{w_\infty}$ be the associated inertia group, and let $\phi_{w_\infty}\in \G/I_{w_\infty}$ denote the Frobenius at $w_\infty$. Let $\mathfrak N(v)$ denote the absolute norm of $v$.

\begin{definition}[{\cite[Definition~VII.10.1]{NeukirchANT}}]
	Let $\chi$ be an irreducible complex character of $\G$ with open kernel (i.e. $\G/\ker\chi$ is finite). Define the $S$-truncated Artin $L$-function as
	\[L_S(\chi,s)\colonequals \prod_{v\notin S} \det \left( 1-\phi_{w_\infty} \mathfrak N(v)^{-s} \middle| V_\chi^{I_{w_\infty}}\right)^{-1},\]
	where $V_\chi$ is the representation affording $\chi$, that is, $V_\chi$ is a simple $\CC[\G/\ker\chi]$-module with character $\chi$; in other words, $V_\chi$ is a $\CC$-vector space, $\rho_\chi: \G\twoheadrightarrow\G/\ker\chi\to \Aut V_\chi$ an Artin representation, and $\chi=\Tr\rho_\chi$. Then $V_\chi^{I_{w_\infty}}$ denotes the subspace fixed by $I_{w_\infty}$. 
\end{definition}
The definition is independent of the choices for the places $w_\infty$: changing $w_\infty$ leads to replacing $\phi_{w_\infty}$ by a conjugate, and the determinant is conjugation invariant. The function $L_S(\chi,s)$ is absolutely convergent for $\Re s>1$, and admits a meromorphic conti\-nuation to the complex plane.

Let $\chi_\cyc:\Gal(K(\mu_{p^\infty})/K)\to \ZZ_p^\times$ be the cyclotomic character. The decomposition $\ZZ_p^\times\simeq \mu_{p-1}\times(1+p\ZZ_p)$ gives rise to characters
\begin{align*}
	\omega: \Gal(K(\mu_{p^\infty})/K) &\to \Gal(K(\zeta_p)/K)\to \mu_{p-1} ;\\
	\kappa: \Gal(K(\mu_{p^\infty})/K) & \to \Gal(K_\infty/K)\to 1+p\ZZ_p .
\end{align*}
Here $\omega$ is the Teichmüller character.

The construction of $p$-adic Artin $L$-functions for linear characters was given by Pi.~Cassou-Nogu\`es resp. Deligne--Ribet in \cite{Cassou1979} resp. \cite{DeligneRibet}. From now on, we assume that $S$ contains all places ramifying in $\LL/K$ (so in particular, all $p$-adic places) as well as all infinite places. Moreover, we fix an isomorphism $\iota: \CC\xrightarrow{\sim}\CC_p$ of fields.
\begin{theorem} \label{thm:CNDR}
	For each linear character $\chi: \G\to\CC_p^\times$ with open kernel, there exists a unique continuous function (called the $S$-truncated $p$-adic Artin $L$-function associated with $\chi$)
	\begin{align*}
		L_{p,S}(\chi,-)&:\ZZ_p\to\CC_p  \phantom{\to\CC_p\to} \text{for $\chi\ne\one$};\\
		L_{p,S}(\one,-)&:\ZZ_p-\{1\}\to\CC_p
	\end{align*}
	satisfying the interpolation property
	\[L_{p,S}(\chi,1-n)= \iota \left( L_S\left(\iota^{-1}\circ \chi\omega^{-n},1-n\right)\right) \quad \forall n\ge1.\]
	Moreover, there exist power series $G_{\chi,S}(X)\in \OO_{\QQ_p(\chi)}[[X]]$ such that
	\[L_{p,S}(\chi,1-s)=\frac{G_{\chi,S}(u^s-1)}{H_\chi(u^s-1)},\]
	where
	$\gamma$ is the fixed topological generator of $\Gamma$, and it  corresponds to $1+X$ under the (non-canonical) isomorphism $\OO_{\QQ_p(\chi)}\llb\Gamma\rrb\simeq\OO_{\QQ_p(\chi)}[[X]]$,
	$u\colonequals\kappa(\gamma)$, and
	\[H_\chi(X)\colonequals\begin{cases}\chi(\gamma)(1+X)-1 & \text{if $H\subseteq \ker\chi$;}\\ 1 & \text{otherwise.}\end{cases}\]
\end{theorem}
Note that due to work of Siegel \cite{Siegel}, the interpolation property does not depend on the choice of the isomorphism $\iota$.

The construction of $p$-adic Artin $L$-functions was extended to not necessarily linear Artin characters via Brauer induction by Greenberg in \cite{Greenberg1983}. This, however, weakens the above statement, with $G_{\chi,S}(X)$ now being a quotient of two power series. The interpolation property is also weakened from $n\ge1$ to $n\ge2$. Both of these weakenings are due to possibly negative coefficients in the Brauer induction theorem. Ellerbrock and Nickel recently verified the interpolation property at $n=1$, see \cite[Corollary~6.3]{EN2022}.
\begin{theorem}
	Let $\chi:\G\to\CC_p$ be a nontrivial {irreducible} character with open kernel. There exists a unique continuous function
	\[L_{p,S}(-,\chi):\ZZ_p\to\CC_p\]
	satisfying the interpolation property
	\[L_{p,S}(1-n,\chi)=\iota\left(L_S\left(1-n,\iota^{-1}\circ\chi\omega^{-n}\right)\right) \quad \forall n\ge1.\]
	Moreover, there exists a fraction $G_{\chi,S}(X)\in \Quot(\OO_{\QQ_p(\chi)}[[X]])$ of power series such that
	\begin{equation} \label{eq:Greenberg-Lp}
		L_{p,S}(1-s,\chi)=\frac{G_{\chi,S}(u^s-1)}{H_\chi(u^s-1)},
	\end{equation}
	where $u$, $T$ and $H_\chi(X)$ are as in \cref{thm:CNDR}.
\end{theorem}

\subsection{Equivariant \texorpdfstring{$p$}{p}-adic Artin \texorpdfstring{$L$}{L}-functions}
Ritter and Weiss defined equivariant $p$-adic Artin $L$-functions as follows; see \cite[\S\S3--4]{TEIT-II} for the original version and \cite[\S\S4.4,~4.8]{NABS} for the formulation used here.
\begin{definition}[{\cite[Theorem~8]{TEIT-II}}]\label{def:Hom-star}
	Let $R_p(\G)$ be the additive group generated by $\QQ_p^\al$-valued characters of $\G$ with open kernel. Let $\Hom^*\!\left(R_p(\G), \Q^\al(\Gamma)^\times\right)$ denote the group of homomorphisms $f:R_p(\G)\to\Q^\al(\Gamma)^\times$ satisfying the following transformation properties: 
	\begin{enumerate}
		\item $f(\chi\otimes\rho)=\rho^\sharp (f(\chi))$ for all characters $\rho$ such that $\res^\G_H\rho=\one$ (so-called  type W characters), where $\rho^\sharp\in\Aut(\Q^\al(\Gamma))$ is the automorphism given by {$\rho^\sharp(g)=\rho(g)g$ for all $g\in \Gamma$};
		\item $f({}^\sigma\chi)=\sigma(f(\chi))$ for all $\sigma\in\Gal(\QQ_p^\al/\QQ_p)$.
	\end{enumerate}
\end{definition}

\begin{theorem}[{\cite[Proof of Theorem~8]{TEIT-II}}] \label{thm:RW-iso}
	There is an isomorphism
	\begin{equation*}
		\cent(\Q(\G))^\times \xrightarrow{\sim} \Hom^*\!\left(R_p(\G), \Q^\al(\Gamma)^\times\right).
	\end{equation*}
\end{theorem}
We give a brief description of the isomorphism in \cref{thm:RW-iso}. Ritter and Weiss proved that there exists a unique element 
\begin{equation}\label{eq:RW-gamma-chi}
	\gamma_\chi\in\cent\left(\Q^\al(\G)e_\chi\right)
\end{equation}
which acts trivially on the space $V_\chi$ and it is of the form $\gamma_\chi=g_\chi c_\chi$, where $g_\chi \in \G$ is mapped to $\gamma^{w_\chi}\in\Gamma$ under the projection $\G\twoheadrightarrow\Gamma$, and $c_\chi\in (\QQ_p^\al[H]e_\chi)^\times$; see \cite[Proposition~5(2)]{TEIT-II}. In fact, this construction also works over a sufficiently large finite extension of $\QQ_p$ instead of $\QQ_p^\al$, see \cite[\S4.4]{NABS}. Let $\Gamma_\chi$ denote the procyclic group generated by $\gamma_\chi$. Then $\cent(\Q^\al(\G)e_\chi)= \Q^\al(\Gamma_\chi)$, which gives rise to a group homomorphism
\begin{equation*}
	j_\chi: \cent\left(\Q^\al(\G)\right)\twoheadrightarrow \cent(\Q^\al(\G)e_\chi)\simeq\Q^\al\left(\Gamma_\chi\right)\xrightarrow{\gamma_\chi\mapsto\gamma^{w_\chi}}\Q^\al\left(\Gamma\right)
\end{equation*}
Then the isomorphism of \cref{thm:RW-iso} is given by 
\begin{equation}\label{eq:j-chi-def}
	z\mapsto (\chi\mapsto j_\chi(z)).
\end{equation}

Characterwise $p$-adic Artin $L$-functions satisfy the transformation properties listed in \cref{def:Hom-star}, see \cite[Proposition~11]{TEIT-II}. Therefore they can be seen as elements on the right hand side of the isomorphism of \cref{thm:RW-iso}, so the following definition makes sense:

\begin{definition}[{\cite[\S4.8]{NABS}}]\label{def:RWPhi}
	Let $\RWL_{\LL/K,S}\in \Hom^*\!\left(R_p(\G), \Q^\al(\Gamma)^\times\right)$ be defined by the assignment
	\[\RWL_{\LL/K,S}: \chi\mapsto \frac{G_{\chi,S}(\gamma-1)}{H_\chi(\gamma-1)}.\]
	The $S$-truncated equivariant $p$-adic Artin $L$-function $\Phi_S(\LL/K)\in \cent(\Q(\G))^\times$ is the corresponding element under the isomorphism of \cref{thm:RW-iso}.
\end{definition}

\subsection{Smoothing}
Recall that $S$ is a finite set of places of $K$ containing all places ramifying in $\LL/K$ as well as all infinite places.
Let $T$ be another finite set of places of $K$ such that $S\cap T=\emptyset$; we shall refer to $T$ as the smoothing set.
We define a smoothed version of the equivariant $p$-adic Artin $L$-function, following \cite[§8.3]{NABS}. The element $\Psi_{S,T}$ there differs from our $\Phi_S^T$ by a twist by the cyclotomic character, which is in line with the fact that $\Psi_{S,T}$ is associated with a CM extension, whereas $\Phi_S^T$ corresponds to the totally real extension $\LL/K$.

\begin{definition}\label{def:PhiST}
	For each $v\in T$ fix a place $w_\infty$ of $\LL$. Define
	\[\Phi_S^T(\LL/K) \colonequals \Phi_S(\LL/K) \cdot \prod_{v\in T} \nr\big(1-\phi_{w_\infty}\big) .\]
	Here the reduced norm map is applied to the class of the map
	\[\Lambda(\G) \xrightarrow{1-\phi_{w_\infty}} \Lambda(\G)\]
	in $K_1(\Q(\G))$. If there is no risk of confusion, we will omit the field extension from the notation and simply write $\Phi_S^T$.
\end{definition}

{We shall address the hypotheses imposed upon $T$ in \cref{rem:T-hyp}; for the results of this section, only disjointness from $S$ is required.}
It would have been possible to introduce smoothing already in the complex $L$-function, as is done in \cite[(0.20) and Footnote~4]{AMG}. Instead, we make the following:

\begin{definition}
	Let $\RWL_{\LL/K,S}^T$ denote the element of $\Hom^*\!\left(R_p(\G), \Q^\al(\Gamma)^\times\right)$ corresponding to $\Phi_S^T$ under the isomorphism of \cref{thm:RW-iso}.
\end{definition}

In the rest of this section, we describe how $\RWL_{\LL/K,S}^T$ and $\Phi_S^T(\LL/K)$ transform under replacing $\G$ by a quotient. {Analogous} properties are available for non-smoothed equivariant $p$-adic Artin $L$-functions, see \cite[Proposition~12]{TEIT-II}. These go back to corresponding statements for non-equivariant $p$-adic Artin $L$-functions, see \cite{Greenberg1983}, and ultimately to characterwise complex Artin $L$-functions, see \cite[Proposition VII\S10.4(iii,iv)]{NeukirchANT}.

Let $\LL/\LL'/K$ be a tower of fields such that $\LL'/K$ satisfies the same conditions as imposed upon $\LL/K$. Then all {of the} notions defined {above} for $\LL/K$ are also defined for $\LL'/K$. We shall distinguish the objects associated with $\LL'/K$ from those defined for $\LL/K$ by a prime symbol, e.g. $\G'\colonequals\Gal(\LL'/K)$. Note that {$\Gamma'=\Gamma=\Gal(K_\infty/K)$}. We choose the places relevant to the construction of smoothed $p$-adic $L$-functions such that they remain compatible with the non-primed versions, that is, we require $w_\infty\mid w_\infty'$. In particular, there is a natural map $\pi:\G\twoheadrightarrow \G'$ sending $\phi_{w_\infty}$ to $\phi_{w_\infty'}$.
The natural projection $\pi$ induces surjective maps $\Lambda(\G)\twoheadrightarrow \Lambda(\G')$ and $\Q(\G)\twoheadrightarrow\Q(\G')$, which we will, by abuse of notation, also denote by $\pi$. Functoriality of $K$-groups induces a morphism $K_1(\pi):K_1(\Q(\G))\to K_1(\Q(\G'))$.

\begin{lemma} \label{RWL-functoriality}
	The elements $\RWL_{\LL/K,S}^T$ and $\Phi_S^T(\LL/K)$ satisfy the following transformation properties:
	in the setup above,
	\begin{align*}
		\RWL_{\LL/K,S}^T \circ\inf_{\G'}^\G  &= \RWL_{\LL'/K,S}^T,\\
		\pi\left(\Phi_S^T(\LL/K)\right)&=\Phi'^T_S(\LL'/K),
	\end{align*}
	where $\inf_{\G'}^\G$ denotes inflation from $\G'$ to $\G$.
\end{lemma}

\begin{proof}
	For $T=\emptyset$, this is \cite[Proposition~12]{TEIT-II}. It remains to verify compatibility of the smoothing factors $(1-\phi_{w_\infty})$ on both sides.
	
	There is a commutative diagram:
	\begin{equation}\label{eq:inflation-diagram}
		\begin{tikzcd}
			K_1(\Q(\G)) \ar[dd, "K_1(\pi)"] \ar[rr] \ar[rd,"\nr"] && \Hom^*\!\left(R_p(\G), \Q^\al(\Gamma)^\times\right) \ar[dd, "-\circ\inf^\G_{\G'}"] \\
			& \cent(\Q(\G))^\times \ar[ru,"\sim"]  \\
			K_1(\Q(\G')) \ar[rr] \ar[rd,"\nr"] && \Hom^*\!\left(R_p(\G'), \Q^\al(\Gamma)^\times\right) \\
			& \cent(\Q(\G'))^\times \ar[ru,"\sim"] \ar[from=2-2, crossing over, "\pi" near start]
		\end{tikzcd}
	\end{equation}
	The square containing $K_1$-groups and $\Hom^*$-groups is \cite[Lemma~9]{TEIT-II} for deflation maps, and the two triangles are \cite[(4.4)]{NABS}, originally formulated in \cite[p.~558]{TEIT-II}. The left front square is obviously commutative. 
	
	Commutativity of the right front square can be seen from the definition of the northeast-pointing arrows: these are given by \eqref{eq:j-chi-def}.
	For commutativity of the square, one only needs to check that $\pi$ maps $\gamma_\chi$ to $\gamma_\chi'$, which follows from the uniqueness property of $\gamma_\chi$, as shown in \cite[Lemma~3.3.1(iii)]{AMG}. 
	
	Finally, the left vertical arrow sends $(1-\phi_{w_\infty})$ to $(1-\phi_{w'_\infty})$ because the places $w_\infty$ and $w_\infty'$ were chosen compatibly. The claims follow.
\end{proof}

In \cref{sec:changing-top-field}, these properties will be used to establish functoriality for integrality of $\Phi_S^T$ under the assumption of the main conjecture.

\begin{remark}
	One may also consider replacing $\G$ by an open subgroup $\G'$. Then in the statement of \cref{RWL-functoriality}, one can replace inflation by induction, as done in \cite[Proposition~12]{TEIT-II}. See also \cite[\S3.3.2]{AMG}. This, however, does not directly lead to a functoriality statement as for quotients, the crux being the definition of a map $\Q(\G)^\times\to \Q(\G')^\times$ making the following square commutative.
	\[\begin{tikzcd}
		\Q(\G)^\times \ar[d,"\nr"] \ar[r, dashed] & \Q(\G')^\times \ar[d,"\nr"] \\
		\cent(\Q(\G))^\times \ar[r] & \cent(\Q(\G'))^\times
	\end{tikzcd}\]
\end{remark}

\section{The equivariant \texorpdfstring{$p$}{p}-adic Artin conjecture} \label{sec:conjectures}
We begin by recalling the complex Artin conjecture. As before, let $S$ be a finite set of places of $K$ containing all places ramifying in $\LL/K$ as well as all infinite places.
\begin{conjecture}[Artin conjecture]
Let $\chi$ be an irreducible character of $\G$ with open kernel. Then the function $L_S(\chi,s)$ is entire whenever $\chi$ is nontrivial.
\end{conjecture}
The assertion has been proven for abelian extensions \cite[Theorem VII.10.6]{NeukirchANT}, but remains open in general. 
The corresponding statement for $p$-adic Artin $L$-functions was conjectured by Greenberg in \cite{Greenberg1983} for the depletion set $S=S_p\cup S_\infty$.
\begin{theorem}[$p$-adic Artin conjecture] \label{padicArtin}
	For a character $\chi:\G\to\QQ_p^\al$ with open kernel,
	\[G_{\chi,S}(X)\in \ZZ_p[[X]]\otimes_{\ZZ_p}\QQ_p^\al,\]
	where $G_{\chi,S}\in \Frac(\OO_{\QQ_p(\chi)}[[X]])$ is the object appearing in the numerator of the expression for the $p$-adic Artin $L$-function in \eqref{eq:Greenberg-Lp}.
\end{theorem}
In \cite[Proposition~5]{Greenberg1983}, Greenberg showed that the appropriate characterwise main conjecture implies the $p$-adic Artin conjecture. Since the main conjecture is known due to the work of Wiles, so is the $p$-adic Artin conjecture, see \cite[Theorems~1.1 and 1.2]{Wiles}.

Greenberg also considered the following stronger version, and proved that it follows from the main conjecture under the additional assumption of the vanishing of Iwasawa's $\mu$-invariant, see \cite[p.~87]{Greenberg1983}. The statement was proven unconditionally by Ritter and Weiss in \cite[Remark~(G)]{TEIT-II}.
\begin{theorem}[refined $p$-adic Artin conjecture] \label{padicArtin2}
	For a character $\chi:\G\to\QQ_p^\al$ with open kernel,
	\[G_{\chi,S}(X)\in \OO_{\QQ_p(\chi)}[[X]].\]
\end{theorem}

In analogy with these conjectures, we posit the following equivariant version of the $p$-adic Artin conjecture. Let $T$ be a finite set of places of $K$ such that $S\cap T=\emptyset$ and $T\ne\emptyset$.
\begin{conjecture}[equivariant $p$-adic Artin conjecture] \label{integrality-Phi}
	Let $\mathfrak M$ be a $\Lambda(\Gamma_0)$-order in $\Q(\G)$ containing $\Lambda(\G)$. Then the smoothed equivariant $p$-adic Artin $L$-function $\Phi_S^T$ is in the image of the composite map
	\begin{equation} \label{nr-composite}
		\mathfrak M\cap \Q(\G)^\times \to K_1(\Q(\G)) \xrightarrow{\nr} \cent(\Q(\G))^\times,
	\end{equation}
	where the first arrow is the natural map sending an invertible element to the class of the $1\times 1$ matrix consisting of said element.
\end{conjecture}

\begin{remark}
	\Cref{integrality-Phi} is a generalisation of a question of Nichifor--Palvannan \cite[Question~1.2]{NichiforPalvannan}, who considered extensions of totally real fields cut out by a non-trivial mod~$p$ character, with the Galois group being a direct product of $\Gamma$ and a finite group. They showed that the answer to their question is positive for a certain maximal order, see \cite[Theorem~1]{NichiforPalvannan}. 
	
	We shall provide further comments on how the equivariant $p$-adic Artin conjecture is analogous to the non-equivariant version in \cref{rem:epAC-AC}.
\end{remark}
\begin{remark}
	The smaller $\mathfrak M$ is, the stronger the statement of \cref{integrality-Phi} becomes. In particular, it is at its strongest resp. weakest when $\mathfrak M$ is $\Lambda(\G)$ resp. a maximal order.
\end{remark}

\section{Main conjectures} \label{sec:MC}
In order to state a smoothed version of the equivariant Iwasawa main conjecture---in particular, to define the object on the algebraic side---we impose further restrictions on our setup. Let $L/K$ be a finite Galois extension of number fields such that $L$ is a CM number field containing $\zeta_p$ (we will address these assumptions in \cref{rem:assumptions-on-L}). Let $L^+$ denote its maximal totally real subfield. Let $L_\infty\colonequals K_\infty L$ and $L_\infty^+\colonequals K_\infty L^+$. In the notation above, set $\LL\colonequals L^+_\infty$. We keep the assumption that $S$ and $T$ are finite non-empty disjoint sets of places of $K$, $S$ contains all places ramifying in $L^+_\infty/K$ and all infinite places.
Then the objects $\Gamma$, $H$, $\G$ and $\Phi_S^T$ are defined as in \cref{sec:smoothed-eq-p-adic-Artin-L}. Finally, let $\widetilde \G\colonequals \Gal(L_\infty/K)$.

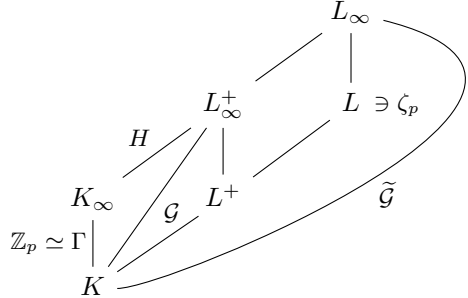
\[\begin{tikzcd}[remember picture, every label/.append style = {font = \small}, every arrow/.append style={no head}]
	&& |[alias=Li]| L_\infty \ar[dl] \ar[dr] \\
	& L_\infty^+ \arrow[dr] \arrow[dd, "\G"] && |[alias=L]| L \ar[dl] \\
	K_\infty \arrow[dr, "\ZZ_p\simeq\Gamma"'] \arrow[ru, "H"] && L^+ \arrow[ld, no head] \\
	& |[alias=K]|K                                                                
\end{tikzcd}\begin{tikzpicture}[overlay,remember picture]
	\draw (K) .. controls +(2.5,0) and +(5.5,-1) .. (Li) node [midway,shift=({.2,-.2})] {\small$\widetilde\G$};
	\node[right=-1mm of L] {\small$\ni\zeta_p$};
\end{tikzpicture}\]

\begin{remark} \label{rem:T-hyp}
    {When working with a smoothing set $T$, one usually requires the condition that the only root of unity $\zeta\in\mu(\LL)$ congruent to $1$ modulo all primes above $T$ is $\zeta=1$ itself; see \cite[Remark~3.1]{NABS} or \cite[p.~4]{OnBrumerStark}. This condition is satisfied whenever $T$ contains two primes of different residue characteristics or a prime with large enough norm.}
    
    {After tensoring with $\ZZ_p$, the condition on $T$ becomes vacuous, provided that $T$ is non-empty.
    Indeed, suppose that $1\ne \zeta\in\mu(\LL)\otimes_\ZZ\ZZ_p$ is a $p$-power root of unity such that for all $v\in T$, we have $\zeta-1\in w_\infty$, where $w_\infty$ is as in \cref{def:PhiST}. Then $w_{\infty}\mid p$, but $T$ contains no $p$-adic places, since those all lie in $S$.}

    {In this work, the objects involving the smoothing set are the equivariant $p$-adic $L$-functions $\Phi^T_S$ introduced above, and the Johnston--Nickel module $Y^T_S$ defined below. Both of these are objects over $\ZZ_p$ and not over $\ZZ$. Therefore no further assumption involving roots of unity is necessary.}
\end{remark}

\subsection{The Johnston--Nickel module \texorpdfstring{$Y_S^T$}{Y\_S\^{}T}} \label{sec:YST}
The object on the algebraic side of the smoothed equivariant main conjecture is given rise to by the module $Y_S^T$, which was defined by Johnston and Nickel \cite{NABS}. We now recall the properties of this module.

Let $j\in\widetilde\G$ denote complex conjugation. Every $\Lambda(\widetilde\G)$-module $M$ decomposes as $M^+\oplus M^-$, where $j$ acts as $+1$ resp. $-1$ on $M^+$ resp. $M^-$. 
Let \[\Lambda\big(\widetilde\G\big)_-\colonequals\Lambda\big(\widetilde \G\big)\big/(1+j).\] The $\Lambda(\widetilde \G)$-module structure on $M$ induces $\Lambda(\G)$- resp. $\Lambda(\widetilde \G)_-$-module structures on the direct summand $M^+$ resp. $M^-$.

For a $\Lambda(\widetilde\G)$-module $M$ and $r\in\ZZ$, let $M(r)$ denote the $r$th Tate twist, which has the same underlying abelian group as $M$ with group action twisted by the $r$th power of the cyclotomic character: $g\star m\colonequals\chi_\cyc(g)^r\cdot g* m$ where $g\in\G$ and $ m\in M$, ``$\star$'' is the new action and ``$*$'' is the old one. We note two properties of $\Lambda(\widetilde \G)$-modules which shall be used in the sequel.

\begin{lemma} \label{lem:minus-Tate}
	Let $M$ be a $\Lambda(\widetilde \G)$-module. Then $(M(-1))^-(1)= M^+$.
\end{lemma}
\begin{proof}
	Since $\chi_{\cyc}(j)=-1$, we have $(M(-1))^-(1)=(M(-1)(1))^+=M^+$.
\end{proof}

\begin{lemma} \label{LambdaG-twist}
	There is an isomorphism $\Lambda(\G)\simeq \Lambda(\widetilde \G)_-(1)$ of $\Lambda(\widetilde \G)$-modules.
\end{lemma}
\begin{proof}
	The ring $\Lambda(\widetilde \G)_-$ has a natural $\Lambda(\widetilde \G)$-module structure via the natural projection $\Lambda(\widetilde \G)\twoheadrightarrow \Lambda(\widetilde \G)_-$. {Therefore} $\Lambda(\widetilde \G)_-$ as a $\Lambda(\widetilde \G)$-module is generated by a single element $m\in \Lambda(\widetilde \G)_-$:
	\begin{align}
		\langle m\rangle_{\Lambda(\widetilde \G)}&=\Lambda(\widetilde \G)_-. \notag\\
		\intertext{Taking Tate twists, we get}
		\label{eq:Lambda-generation}
		\langle m(1)\rangle_{\Lambda(\widetilde \G)}&=\Lambda(\widetilde \G)_-(1).
	\end{align}
	
	The element $j\in\widetilde \G$ acts as $-1$ on $\Lambda(\widetilde \G)_-=\Lambda(\widetilde \G)/(1+j)$. Since $\chi_{\cyc}(j)=-1$, this means that after a Tate twist, $j$ acts trivially on $\Lambda(\widetilde \G)_-(1)$. Since $\G=\widetilde \G/\langle j\rangle$, it follows that the action of $\Lambda(\widetilde \G)$ on $\Lambda(\widetilde \G)_-$ factors through $\Lambda(\G)$.
	
	Consequently, \eqref{eq:Lambda-generation} can be sharpened to $\langle m(1)\rangle_{\Lambda(\G)}=\Lambda(\widetilde \G)_-(1)$, which means that there is a surjection of $\Lambda(\G)$-modules
	\begin{equation} \label{eq:Lambda-generation-2}
		\Lambda(\G)\twoheadrightarrow\Lambda(\widetilde\G)_-(1).
	\end{equation}
	The group $\Gamma_0$ is a central open subgroup both in $\G$ and in $\widetilde\G$, and the ranks of $\Lambda(\G)$ and $\Lambda(\widetilde\G)_-(1)$ as $\Lambda(\Gamma_0)$-modules agree. Therefore \eqref{eq:Lambda-generation-2} is an isomorphism.
\end{proof}

Let $X_S$ denote the $S$-ramified Iwasawa module associated with $L_\infty$. More explicitly, $X_S$ is the Galois group of the maximal abelian pro-$p$ extension of $L_\infty$ unramified outside $S$ over $L_\infty$.
The maximal $j$-invariant submodule of $X_S$ is the $S$-ramified Iwasawa module $X_S^+$ associated with $L_\infty^+$. This is a finitely generated torsion $\Lambda(\Gamma_0)$-module by \cite[Proposition~11.3.1]{NSW}. For each $v\in T$, we fix places $w_\infty$ resp. $\widetilde w_\infty$ of $L^+_\infty$ resp. $L_\infty$ such that $w_\infty\mid\widetilde w_\infty$.

\begin{theorem}[{\cite[Proposition~8.5, Lemma~8.6]{NABS}}]
	There exists a $\Lambda\big(\widetilde\G\big)_-$-module $Y_S^T(-1)$ sitting in an exact sequence of $\Lambda(\widetilde\G)_-$-modules
	\begin{equation} \label{eq:YST-def}
		0\to X_S^+(-1) \to Y_S^T(-1) \to \left(\bigoplus_{v\in T} \ind_{\widetilde\G_{\widetilde w_\infty}}^{\widetilde\G} \ZZ_p(-1) \right)^-\to\ZZ_p(-1)\to0.
	\end{equation}
	Moreover, $Y_S^T(-1)$ has projective dimension {at most} $1$ over $\Lambda(\widetilde \G)_-$, and it is torsion over $\Lambda(\Gamma_0)$.
\end{theorem}
\begin{remark} \label{rem:YST-T}
	Note that our assumptions on $S$ and $T$ are sufficient for this result. Indeed, the proof relies only upon Assumptions~(i--iv) on \cite[1254]{NABS}; our setup covers (i--iii), and (iv) follows from the fact that $T$ is non-empty and it contains no $p$-adic places. {See also \cref{rem:T-hyp}.}
\end{remark}

Since the module $Y_S^T(-1)$ has a projective resolution of length one over $\Lambda(\widetilde \G)_-$, there exist a projective $\Lambda\big(\widetilde\G\big)_-$-module $P$ and an integer $n\ge0$ such that there is an exact sequence
\begin{equation} \label{eq:proj-res-1}
	0\to P \xrightarrow{} \left(\Lambda\big(\widetilde\G\big)_-\right)^n\to Y_S^T(-1)\to0.
\end{equation}
Since $Y_S^T(-1)$ is torsion over $\Lambda(\Gamma_0)$, the difference of the $K_0$-classes over $\Lambda(\widetilde\G)_-$ of the first two terms in this short exact sequence becomes zero over the ring of quotients $\Q(\widetilde\G)_-$:
\begin{align*}
	K_0\left(\Lambda\big(\widetilde\G\big)_-\right) &\hookrightarrow K_0\left(\Q\big(\widetilde\G\big)_-\right), \\
	\left[\Lambda\big(\widetilde\G\big)_-^n\right]-[P] &\mapsto 0.
\end{align*}
Here the map is injective by a result of Witte, see \eqref{Witte}.
Hence the first two terms of the projective resolution \eqref{eq:proj-res-1} are stably isomorphic over $\Lambda(\widetilde\G)_-$. That is, there is a projective $\Lambda(\widetilde\G)_-$-module $Q$ such that $\Lambda\big(\widetilde\G\big)_-^n\oplus Q\simeq P\oplus Q$. 

Since projective modules are direct summands of free modules, by possibly enlarging $Q$, there are isomorphisms
$\Lambda\big(\widetilde\G\big)_-^m\simeq \Lambda\big(\widetilde\G\big)_-^n\oplus Q\simeq P\oplus Q$ for some $m\ge0$. Consequently, there is a projective resolution
\[0\to \left(\Lambda\big(\widetilde\G\big)_-\right)^m \xrightarrow{} \left(\Lambda\big(\widetilde\G\big)_-\right)^m\to Y_S^T(-1)\to0.\]
Twist by $(+1)$, using \cref{LambdaG-twist}:
\begin{equation} \label{eq:YST-resolution}
	0\to \Lambda(\G)^m \xrightarrow{\alpha} \Lambda(\G)^m\to Y_S^T\to0,
\end{equation}
where $\alpha$ is some $m\times m$ matrix. This allows us to associate with $Y_S^T$ a class in the relative $K$-group $K_0(\Lambda(\G),\Q(\G))$: considering $[Y_S^T]$ as a complex concentrated in degree $0$, the resolution \eqref{eq:YST-resolution} shows that in the derived category $\DC(\Lambda(\G))$, this is isomorphic to 
\[
\left[Y_S^T\right]=\Bigcomplex{\Lambda(\G)^m}{-1}{\alpha}{\Lambda(\G)^m}{0},
\]
where the complex on the right is concentrated in degrees $-1$ and $0$. In particular, $[Y_S^T]$ is in $\DCpt(\Lambda(\G))$, and thus it defines a class in $K_0(\Lambda(\G),\Q(\G))$.

\subsection{The smoothed equivariant main conjecture}
Let us first recall the equivariant Iwasawa main conjecture (EIMC) with uniqueness and without smoothing, as formulated in \cite[Conjecture~4.3]{NABS} and \cite[Conjecture~5.4]{UAEIMC}. The original formulations are due to Ritter--Weiss \cite{RW-MC} and Kakde \cite[Theorem~2.11]{Kakde}; we refer to \cite{Venjakob-RWK} for a comparison of these approaches.

\begin{definition} Consider the following complex:
	\[C_S^\bullet(L_\infty^+/K)\colonequals \mathrm{R}\!\Hom\left( \mathrm{R}\Gamma_\eet\left(\Spec \OO_{L_\infty^+,S}, \underline{\QQ_p/\ZZ_p}\right), \QQ_p/\ZZ_p\right),\]
	where $\OO_{L_\infty^+,S}$ is the ring of $S$-integers of $L_\infty^+$, consisting of elements of $L_\infty^+$ with non-negative valuation away from $S$, and $\underline{\QQ_p/\ZZ_p}$ is a constant sheaf.
\end{definition}

The nonzero cohomology groups of the complex $C_S^\bullet(L_\infty^+/K)$ are $H^{-1}(C_S^\bullet(L_\infty^+/K))\simeq X_S^+$ and $H^0(C_S^\bullet(L_\infty^+/K))\simeq \ZZ_p$.
Thus $C_S^\bullet(L_\infty^+/K)$ is a complex in $\DCpt(\Lambda(\G))$. In particular, it defines a class in $K_0(\Lambda(\G),\Q(\G))$.

\begin{conjecture}[non-smoothed EIMC with uniqueness] \label{EIMCu}
	Let $L/K$ be any finite Galois CM extension with $S$ a finite set of places as above.
	There exists a unique element $\zeta_S(L_\infty^+/K)\in K_1(\Q(\G))$ such that $\nr\left(\zeta_S(L_\infty^+/K)\right)=\Phi_S(L^+_\infty/K)$. Moreover,  $\partial\left(\zeta_S(L_\infty^+/K)\right)=-[C_S^\bullet(L_\infty^+/K)]$, where $[C_S^\bullet(L_\infty^+/K)]$ is the class of $C_S^\bullet(L_\infty^+/K)$ in $K_0(\Lambda(\G),\Q(\G))$.
	\[
	\begin{tikzcd}[every label/.append style = {font = \small}]
		& K_1\big(\Q(\G)\big) \arrow[rr, "\partial"]  &  & {K_0(\Lambda(\G),\Q(\G))} \\
		\exists!\, \zeta_S(L_\infty^+/K) \arrow[d, maps to] \arrow[rrr, maps to] & &  & -\left[C_S^\bullet(L_\infty^+/K)\right] \\
		\Phi_S(L^+_\infty/K) & \cent(\Q(\G))^\times  \arrow[from=1-2, "\nr" near start, crossing over] &  &
	\end{tikzcd}
	\]
\end{conjecture}

\begin{remark}
	It is known that the boundary homomorphism is surjective, see \eqref{Witte}.
	Vanishing of $\SK_1(\Q(\G))=\ker(\nr)$ is equivalent to the uniqueness part of the conjecture.
\end{remark}

According to \cite[Proposition~8.5]{NABS}, the middle two terms of \eqref{eq:YST-def}, considered as $0$th resp. $(-1)$st terms of a complex, represent $C_S^\bullet(L_\infty^+/K)(-1)$:
\begin{equation*}
	\Hugecomplex{Y_S^T(-1)}{-1}{}{\left(\displaystyle\bigoplus_{v\in T} \ind_{\widetilde\G_{\widetilde w_\infty}}^{\widetilde\G} \ZZ_p(-1) \right)^-}{0}=\left[C_S^\bullet(L_\infty^+/K)(-1)\right]\in K_0\left(\Lambda\big(\widetilde \G\big)_-,\Q\big(\widetilde \G\big)_-\right).
\end{equation*}
Tate twisting this by $+1$ yields the following, using \cref{lem:minus-Tate,LambdaG-twist}:
\begin{equation} \label{eq:C-bullet-rep}
	\Hugecomplex{Y_S^T}{-1}{}{\left(\displaystyle\bigoplus_{v\in T} \ind_{\widetilde\G_{\widetilde w_\infty}}^{\widetilde\G} \ZZ_p \right)^+}{0}=\left[C_S^\bullet(L_\infty^+/K)\right]\in K_0\left(\Lambda(\G),\Q(\G)\right).
\end{equation}
Note that for all $v\in T$, we have $\left(\ind_{\widetilde\G_{\widetilde w_\infty}}^{\widetilde\G} \ZZ_p \right)^+\simeq \ind_{\G_{w_\infty}}^{\G} \ZZ_p$. Moreover, there is a short exact sequence
\begin{equation} \label{eq:Lemma-8.4-plus}
	0\to \Lambda\left(\G_{w_\infty}\right) \xrightarrow{1-\phi_{w_\infty}} \Lambda\left(\G_{w_\infty}\right) \to \ZZ_p\to0,
\end{equation}
where the labelled arrow is right multiplication by $\left(1-\phi_{\widetilde w_\infty}\right)$.
These factors correspond exactly to those by which we perturbed the analytic $p$-adic $L$-function in \cref{def:PhiST}.
In light of this, we posit the following smoothed version of \cref{EIMCu}:

\begin{conjecture}[smoothed EIMC with uniqueness] \label{sEIMCu} 
	Let $L/K$ be as above, and let $S$ and $T$ be finite non-empty disjoint sets of places of $K$, with $S$ containing all places ramifying in $L^+_\infty/K$ and all infinite places.
	Then there is a unique element $\zeta_S^T(L_\infty^+/K)\in K_1(\Q(\G))$ such that $\nr\left(\zeta_S^T(L_\infty^+/K)\right)=\Phi_S^T$. Moreover, $\partial\left(\zeta_S^T(L_\infty^+/K)\right)=[Y_S^T]$.
	\[
	\begin{tikzcd}[every label/.append style = {font = \small}]
		& K_1\big(\Q(\G)\big) \arrow[rr, "\partial"]  &  & {K_0(\Lambda(\G),\Q(\G))} \\
		\exists!\, \zeta_S^T \arrow[d, maps to] \arrow[rrr, maps to] & &  & \left[Y_S^T\right] \\
		\Phi_S^T(L_\infty^+/K) & \cent(\Q(\G))^\times  \arrow[from=1-2, "\nr" near start, crossing over] &  &
	\end{tikzcd}
	\]
\end{conjecture}

\begin{lemma} \label{EIMCu-implies-sEIMCu}
	Let $L/K$, $S$ and $T$ be as in \cref{sEIMCu}. Then \cref{EIMCu} is equivalent to \cref{sEIMCu}.
\end{lemma}
\begin{proof}
	The ingredients of the proof are all contained in the paragraphs preceding \cref{sEIMCu}.	
	We first deal with the existence and ``moreover'' parts of the two main conjectures. Assume that \cref{EIMCu} holds. Then let 
	\[\zeta_S^T\colonequals \zeta_S\cdot\prod_{v\in T} \left[(1-\phi_{w_\infty})\right]\in K_1(\Q(\G)).\]
	Its reduced norm is
	\[\nr(\zeta_S^T)=\nr(\zeta_S)\cdot\prod_{v\in T}\nr(1-\phi_{w_\infty})=\Phi_S\cdot\prod_{v\in T}\nr(1-\phi_{w_\infty})=\Phi_S^T,\]
	as desired. We compute the image under the boundary homomorphism $\partial$.
	\begingroup \allowdisplaybreaks\begin{align*}
		\partial\left(\zeta_S^T\right)
		&=\partial\left(\zeta_S\right)+\sum_{v\in T}\partial\left(1-\phi_{w_\infty}\right) \\
		&=-\left[C^\bullet(L_\infty^+/K)\right]+\sum_{v\in T}\partial\left(1-\phi_{w_\infty}\right) \\
		\intertext{The equality \eqref{eq:C-bullet-rep} in $K_0(\Lambda(\G),\Q(\G))$ {allows us to} rewrite the first term as the complex represented by ${Y_S^T} \to {\left(\bigoplus_{v\in T} \ind_{\widetilde\G_{w_\infty}}^{\widetilde\G} \ZZ_p \right)^+}$. {By \eqref{eq:C-bullet-rep}, this can be written as follows.}}
		&=\left[Y_S^T\right]-\Hugecompl{\displaystyle\bigoplus_{v\in T} \ind_{\widetilde\G_{\widetilde w_\infty}}^{\widetilde\G} \ZZ_p }{0}+\sum_{v\in T}\partial\left(1-\phi_{w_\infty}\right) \\
		\intertext{We rewrite the sum using the definition of $\partial$ and the isomorphism \eqref{eq:K0-rel-K0-iso}.}
		&=\left[Y_S^T\right]-\sum_{v\in T}\Hugecompl{\ind_{\widetilde\G_{\widetilde w_\infty}}^{\widetilde\G} \ZZ_p }{0}+\sum_{v\in T} \Bigcomplex{\Lambda(\G)}{-2}{1-\phi_{w_\infty}}{\Lambda(\G)}{-1}
	\end{align*}\endgroup
	Using the exact sequence \eqref{eq:Lemma-8.4-plus}, we find that the two summations cancel each other out, and we get $\partial\left(\zeta_S^T\right)=\left[Y_S^T\right]$, as desired. In conclusion, \cref{EIMCu} implies \cref{sEIMCu} without uniqueness. By the same computation, \cref{sEIMCu} implies \cref{EIMCu} without uniqueness by setting 
	\[\zeta_S\colonequals \zeta_S^T\cdot\prod_{v\in T} \left[(1-\phi_{w_\infty})\right]^{-1}\in K_1(\Q(\G)).\]
	
	Lastly, the two uniqueness statements are equivalent to each other, because they are both equivalent to $\SK_1(\Q(\G))=1$.
\end{proof}

\subsection{The equivariant \texorpdfstring{$p$}{p}-adic Artin conjecture for \texorpdfstring{$\zeta_S^T$}{zeta\_S\^T}}

We also formulate an integrality conjecture regarding the element $\zeta_S^T$, in the setup of \cref{sEIMCu}.
\begin{conjecture}\label{integrality-zeta}
	Assume \cref{sEIMCu} without assuming uniqueness; let $\zeta_S^T\in K_1(\Q(\G))$ such that $\nr\zeta_S^T=\Phi_S^T$ and $\partial(\zeta_S^T)=[Y_S^T]$. Let $\mathfrak M$ be a $\Lambda(\Gamma_0)$-order in $\Q(\G)$ containing $\Lambda(\G)$. Then $\zeta_S^T$ is in the image of the natural map
	\[\mathfrak M\cap \Q(\G)^\times \to K_1(\Q(\G)).\]
\end{conjecture}
\begin{remark}
	\Cref{integrality-zeta} is formulated without assuming uniqueness. However, if it holds for one $\zeta^T_S$ satisfying the conditions in the main conjecture, then it holds for all such elements. Indeed, let $\zeta^T_S,\zeta'^T_S \in K_1(\Q(\G))$ such that $\partial(\zeta_S^T)=\partial(\zeta'^T_S)$. Using the localisation sequence \eqref{Witte}, we find a class $[\beta]\in K_1(\Lambda(\G))$ such that $\zeta'^T_S=\zeta^T_S\cdot[\beta]$. Then Vaserstein's theorem \eqref{eq:Dieudonne-det-iso} allows us to find a representative $b\in\Lambda(\G)$ of $\det\beta$ in $\Lambda(\G)$. So if $x\in \mathfrak M\cap\Q(\G)^\times$ {is} a preimage of $\zeta^T_S$ under the natural map in the conjecture, then $\zeta'^T_S$ is the image of $xb$. We will encounter the same argument again in the proof of \cref{thm:integrality}.
\end{remark}

\begin{lemma} \label{lem:integrality-equivalence}
	Assume \cref{sEIMCu} without uniqueness. Then the following hold:
	\begin{enumerate}[label=(\roman*),ref=\roman*]
		\item \label{item:integrality-equivalence-1} \Cref{integrality-zeta} implies \cref{integrality-Phi}.
		\item \label{item:integrality-equivalence-2} If in addition, we assume uniqueness {in the main conjecture}, then \cref{integrality-Phi} implies \cref{integrality-zeta}.
	\end{enumerate}
\end{lemma}
\begin{proof}
	(\ref{item:integrality-equivalence-1}) Let $\zeta_S^T$ be a preimage of $\Phi_S^T$. If this is in the image of the natural map, then $\nr(\zeta_S^T)=\Phi_S^T$ is in the image of the composite map \eqref{nr-composite}, so \cref{integrality-zeta} implies \cref{integrality-Phi}. 
	
	(\ref{item:integrality-equivalence-2}) For the converse, let $\zeta_S^T$ be the unique preimage of $\Phi_S^T$ under the reduced norm map, and let $x\in \mathfrak M\cap\Q(\G)^\times$ be a preimage of $\Phi_S^T$ under the map \eqref{nr-composite}. Then the class $[(x)]\in K_1(\Q(\G))$ is also a preimage of $\Phi_S^T$ under the reduced norm, therefore $[(x)]=\zeta_S^T$ by uniqueness of $\zeta_S^T$.
\end{proof}

\begin{remark}
	One can also consider {the following} weakening of the above conjectures without involving an order $\mathfrak M$.
	That is, one may ask whether $\Phi_S^T$ comes from a $1\times 1$ matrix over $\Q(\G)$, or more explicitly, whether it is in the image of the map
	\begin{equation} \label{eq:composite-Q}
		\Q(\G)^\times\to K_1(\Q(\G))\xrightarrow{\nr} \cent(\Q(\G))^\times.
	\end{equation}
	This is not known in general. {In fact}, it is true that $\Phi_S^T$ is in the image of \eqref{eq:composite-Q} if and only if it is in the image of $\nr: K_1(\Q(\G))\to \cent(\Q(\G))^\times$. 
    {(In particular, $\Phi_S^T$ is in the image of \eqref{eq:composite-Q} if \cref{sEIMCu} holds.)}
    {To verify the claimed equivalence, consider the following commutative diagram coming from \eqref{eq:nr-det}:}
	\[
	\begin{tikzcd}
		\Q(\G)^\times \arrow[rrrdd, "\nr_{\Q(\G)/\cent(\Q(\G))}", bend left] \arrow[rd, two heads] \arrow[rdd, two heads, bend right] &  &  & \\
		& K_1(\Q(\G)) \arrow[d, "\det"', "\sim"] \arrow[rrd, "\nr"] &  & \\
		& \left(\Q(\G)^\times\right)^\ab \arrow[rr, "\nr_{\Q(\G)/\cent(\Q(\G))}"'] &  & \cent(\Q(\G))^\times
	\end{tikzcd}
	\]
	The vertical arrow is an isomorphism by \eqref{eq:Dieudonne-det-iso}. Since the map $\Q(\G)^\times\twoheadrightarrow \left(\Q(\G)^\times\right)^\ab$ is surjective, so is the natural map $\Q(\G)^\times\to K_1(\Q(\G))$ by commutativity of the left triangle. The claimed equivalence follows.
\end{remark}

\subsection{A maximal order main conjecture}
The {so-called} ``maximal order main conjecture'' (a theorem of Ritter and Weiss) admits a smoothed version. We first recall the statement in the non-smoothed case.
\begin{theorem}[{\cite[Theorem~16, Remark~(H)]{TEIT-II},  \cite[\S4.5]{hybrid}}] \label{MOMC}
	Let $\mathfrak M$ be a maximal $\Lambda(\Gamma_0)$-order in $\Q(\G)$ containing $\Lambda(\G)$. If $x_S\in K_1(\Q(\G))$ is such that $\partial(x_S)=-\left[C_S^\bullet(L_\infty^+/K)\right]$, then $\nr(x_S)\Phi_S^{-1}\in \cent(\mathfrak M)^\times$.
\end{theorem}
\begin{corollary}[smoothed maximal order main conjecture] \label{smoothed-MO-EIMC}
	Let $\mathfrak M$ be a maximal $\Lambda(\Gamma_0)$-order in $\Q(\G)$ containing $\Lambda(\G)$. If $x_S^T\in K_1(\Q(\G))$ is such that $\partial(x_S^T)=\left[Y_S^T\right]$, then $\nr(x_S^T)(\Phi_S^T)^{-1}\in \cent(\mathfrak M)^\times$.
\end{corollary}
\begin{proof}
	Let $x_S^T$ be as in the statement, and let
	\[x_S\colonequals x_S^T\cdot \prod_v[(1-\phi_{w_\infty})]^{-1}\in K_1(\Q(\G)).\]
	Then the same argument as in the proof of \cref{EIMCu-implies-sEIMCu} shows that
	\[\partial(x_S)=\partial(x^T_S)+\sum_{v\in T}\partial([1-\phi_{w_{\infty}}]^{-1})=[Y^T_S]-\sum_{v\in T}\partial(1-\phi_{w_\infty})=-\left[C_S^\bullet(L_\infty^+/K)\right].\]
	Consequently, $\nr(x^T_S)(\Phi_S^T)^{-1} = \nr(x_S ) (\Phi_S)^{-1} \in \cent(\mathfrak M)^\times$ by \cref{MOMC}.
\end{proof}

\begin{corollary} \label{cor:Phi-M}
	Let $\mathfrak M$ be a maximal $\Lambda(\Gamma_0)$-order in $\Q(\G)$ containing $\Lambda(\G)$. Then $\Phi_S^T\in \cent(\mathfrak M)$.
\end{corollary}
\begin{proof}
	Let $\alpha\in M_m(\Lambda(\G))$ be as in \eqref{eq:YST-resolution}. On the one hand, since $M_m(\Lambda(\G))\subseteq M_m(\mathfrak M)$, and the latter is a maximal $\Lambda(\G)$-order in $M_m(\Q(\G))$ by \cref{MO8.7}, the reduced norm $\nr(\alpha)$ is contained in $\mathfrak M$ by \cref{MO10.1}.
	On the other hand, $\partial([(\alpha)])=[Y_S^T]$, and so  \cref{smoothed-MO-EIMC} is applicable: $\nr(\alpha)\cdot (\Phi_S^T)^{-1}\in \cent(\mathfrak M)^\times$. The assertion follows.
\end{proof}

\begin{remark}
	The integrality statement of \cref{cor:Phi-M} is unconditional. It can be seen as a special case of the equivariant $p$-adic Artin conjecture.
	Indeed, let $\mathfrak M$ be a $\Lambda(\Gamma_0)$-order in $\Q(\G)$ containing $\Lambda(\G)$. Assuming \cref{integrality-Phi} for this $\mathfrak M$, we obtain that $\Phi_S^T\in \cent(\mathfrak M)$ because of \cref{MO10.1}.
\end{remark}

\begin{remark} \label{rem:epAC-AC}
	We demonstrate how the equivariant $p$-adic Artin conjecture (\cref{integrality-Phi}) is analogous to its non-equivariant counterparts. Let $\mathfrak M$ be a maximal $\Lambda(\Gamma_0)$-order in $\Q(\G)$ containing $\Lambda(\G)$, and assume the validity of \cref{integrality-Phi} for this $\mathfrak{M}$ (see \cref{sec:integrality} for {cases in which this is known to hold}). Then \cref{cor:Phi-M} shows that $\Phi^T_S\in\cent(\mathfrak M)$. Let $\ZZ_p^\al$ denote the integral closure of $\ZZ_p$ in $\QQ_p^\al$. The map in \cref{thm:RW-iso} induces an isomorphism $\cent(\mathfrak M)\simeq \Hom^*(R_p(\G), \ZZ_p^\al\otimes_{\ZZ_p} \Lambda(\Gamma))$, as shown by Ritter--Weiss in \cite[Remark~(H)]{TEIT-II}. Consulting \cref{def:PhiST}, one sees that this map sends $\Phi^T_S$ to the map given by
	\[\chi\mapsto \frac{G_{\chi,S}(\gamma-1)}{H_\chi(\gamma-1)}\cdot\delta_T(\gamma-1),\]
	where 
    \[\delta_T(\gamma-1)=\prod_{v\in T} \det\left(1-\phi_{w_\infty}\,\middle|\,\Hom_{\QQ^\al_p[H]}(V_\chi,\Q(\G))\right),\] 
    see \cite[p.~558]{TEIT-II} or \cite[\S4.5]{NABS}. It follows that $G_{\chi,S}(\gamma-1)$ is integral up to the factor $\delta_T(\gamma-1)/H_\chi(\gamma-1)$, which can be made explicit by choosing the smoothing set $T$ appropriately, {as we will now show}.
	
	Indeed, consider the case of a type~W character, that is, when $\ker\chi\supseteq H$, and assume that $\chi$ is non-trivial. In this case, $\chi$ factors through a finite abelian group, so in particular, it is linear. Using Tchebotaryov's density theorem, one finds a place $v$ of $K$ such that $\phi_{w_\infty}$ is conjugate to $\gamma$ in $\G$. Set $T\colonequals\{v\}$. Then \[\delta_T(\gamma-1)=\det\left(1-\gamma\,\middle|\,\Hom_{\QQ^\al_p[H]}(V_\chi,\Q(\G))\right)=1-\gamma\chi(\gamma^{-1}).\]
	Indeed, if $f\in\Hom_{\QQ^\al_p[H]}(V_\chi,\Q(\G))$, then \[((1-\gamma)f)(v)=f(v)-(\gamma f)(v)=f(v)-\gamma\cdot f(\gamma^{-1} v)=f(v)-\gamma\cdot f\left(\chi(\gamma^{-1} v)\right) = \left(1-\gamma \chi(\gamma)^{-1}\right)\cdot f(v).\]
	Since $\chi$ is of type W, we have $H_\chi(\gamma-1)=\chi(\gamma)\gamma-1$. Passing to power series by replacing $\gamma\mapsto 1+X$, we find that $\delta_T(X)/H_\chi(X)$ has constant term
	\[\frac{1-\chi(\gamma)^{-1}}{\chi(\gamma)-1} = \frac{\chi(\gamma)\cdot \left(1-\chi(\gamma)^{-1}\right)}{\chi(\gamma)\cdot\left(\chi(\gamma)-1\right)}=\frac{1}{\chi(\gamma)}.\]
	The denominator on the left hand side is nonzero because $\chi$ is non-trivial.
	Since $\chi(\gamma)$ is a $p$-power root of unity, the constant term {of $\delta_T(X)/H_\chi(X)$} is invertible, and thus $\delta_T(X)/H_\chi(X)$ is a unit in $\ZZ_p^\al[[X]]$. Hence $G_{\chi,S}(X)\in\ZZ_p^\al[[X]]$, and since $\chi$ is invariant under $\Gal(\QQ_p^\al/\QQ_p(\chi))$, this even shows $G_{\chi,S}(X)\in \OO_{\QQ_p(\chi)}[[X]]$, as stated by the refined $p$-adic Artin conjecture (\cref{padicArtin2}).
\end{remark}

\section{Behaviour of the equivariant \texorpdfstring{$p$}{p}-adic Artin conjecture under changing \texorpdfstring{$S$}{S}, \texorpdfstring{$T$}{T}, and \texorpdfstring{$\G$}{G}} \label{sec:generalities-on-integrality} \label{sec:changing-top-field}
We note some generalities regarding \cref{integrality-Phi}. 
We first describe how the conjecture behaves under adding more places to either $S$ or $T$. 
We will work with a fixed extension $\LL/K$ as in \cref{sec:smoothed-eq-p-adic-Artin-L}. For the sake of brevity, we will suppress the extension in our notation, and simply write $\Phi_S^T\colonequals \Phi_S^T(\LL/K)$. {Recall that $S$ and $T$ are finite sets of places of $K$, with $S$ containing all places ramifying in $\LL/K$ as well as all infinite places, and $T$ being nonempty and disjoint to $S$.}

\begin{lemma} \label{claim:independence-of-T}
	Let $\mathfrak M\subset \Q(\G)$ be a $\Lambda(\Gamma_0)$-order in $\Q(\G)$ containing $\Lambda(\G)$. Let $S$ and $T$ satisfy the hypotheses above, and let $v^*$ be a place of $K$ not in $S\cup T$. If $\Phi_S^T$ is in the image of the map \eqref{nr-composite}, then so is $\Phi_S^{T\cup \{v^*\}}$.
\end{lemma}
\begin{proof}
	By \cref{def:PhiST}, the relationship between the two smoothed equivariant $p$-adic $L$-functions is $\Phi_S^{T\cup\{v^*\}}=\Phi_S^T\cdot \nr\big(1- \phi_{w^*_\infty}\big)$ where $w^*_\infty$ is the (previously) fixed place of $L_\infty^+$ above $v^*$.
	The first term is in the image of the map by assumption, while the second one is the reduced norm of the class of the map given by left multiplication by the element $\big(1-\phi_{w^*_\infty}\big)\in\Lambda(\G)$. This element is invertible in $\Q(\G)$, and $\Lambda(\G)\subseteq\mathfrak M$, hence the claim.
\end{proof}

\begin{lemma} \label{claim:independence-of-S}
	Retain the assumptions of \cref{claim:independence-of-T}. Then if $\Phi_S^T$ is in the image of the map \eqref{nr-composite}, then so is $\Phi_{S\cup \{v^*\}}^{T}$.
\end{lemma}
\begin{proof}
	Noting that $\chi_{\cyc}(\phi_{w^*_\infty})=\mathfrak N(v^*)$, {where $\mathfrak N$ denotes the absolute norm}, we have $\Phi^T_{S\cup\{v^*\}}=\Phi^T_S\cdot\nr\left(1-\phi_{w^*_\infty} \chi_{\cyc}(\phi_{w^*_\infty})^{-1}\right)$. Indeed, this can be seen e.g. by comparing Proposition~3.2.13 and Lemma~4.2.1 of \cite{AMG}. Then the same argument as in \cref{claim:independence-of-T} works.
\end{proof}

Consider the setup for changing the top field in \cref{RWL-functoriality}.
In particular, if $L/L'/K$ is a tower of fields such that $L'/K$ and $L/K$ both satisfy the assumptions at the beginning of \cref{sec:MC}, then  $\LL=L^+_\infty$ and $\LL'=L'^+_\infty$ fit into this setup.

There is a natural projection $\pi:\G\twoheadrightarrow\G'$ which induces $\Q(\G)\twoheadrightarrow\Q(\G')$ and $K_1(\pi):K_1(\Q(\G))\to K_1(\Q(\G'))$.
Let $\mathfrak M$ be a $\Lambda(\Gamma_0)$-order in $\Q(\G)$ containing $\Lambda(\G)$, and let $\mathfrak M'\colonequals\pi(\mathfrak M)$. Then by surjectivity of the natural maps, $\mathfrak M'$ is a $\Lambda(\Gamma_0)$-order in $\Q(\G')$ containing $\Lambda(\G')$. Write $\Phi_S^T\colonequals \Phi_S^T(\LL/K)$ and $\Phi'^T_S\colonequals\Phi_S^T(\LL'/K)$ for brevity.

\begin{proposition} \label{changing-top-field}
	If $\Phi_S^T$ is in the image of the composite map 
	\[\mathfrak M\cap \Q(\G)^\times\to K_1(\Q(\G))\xrightarrow{\nr}\cent(\Q(\G))^\times,\]
	then $\Phi'^T_S$ is in the image of the composite map
	\[\mathfrak M'\cap \Q(\G')^\times\to K_1(\Q(\G'))\xrightarrow{\nr}\cent(\Q(\G'))^\times.\]
\end{proposition}
\begin{proof}
	The maps in the statement fit into a commutative diagram:
	\[\begin{tikzcd}
		\mathfrak M\cap\Q(\G)^\times \ar[r] \ar[d,"\pi"] & K_1(\Q(\G)) \ar[d, "K_1(\pi)"] \ar[r,"\nr"] & \cent(\Q(\G))^\times \ar[d,"\pi"] \\
		\mathfrak M'\cap\Q(\G')^\times \ar[r] & K_1(\Q(\G')) \ar[r,"\nr"] & \cent(\Q(\G'))^\times
	\end{tikzcd}\]
	By \cref{RWL-functoriality}, we have $\pi(\Phi_S^T)=\Phi'^T_S$. 
	Hence if $\Phi_S^T$ is the image of some $x\in \mathfrak M\cap \Q(\G)^\times$ under the top row, then $\Phi'^T_S$ is the image of $\pi(x)$ under the bottom row.
\end{proof}

\begin{remark}
	Using the same methods as in the proof of \cref{lem:integrality-equivalence}, the results of \cref{claim:independence-of-T,claim:independence-of-S,changing-top-field} can be reformulated for $\zeta_S^T$ under the assumption of \cref{sEIMCu}.
\end{remark}

\begin{remark} \label{rem:assumptions-on-L}
	In order to formulate the smoothed equivariant Iwasawa main conjecture (\cref{sEIMCu}), we need to assume that $L$ is a CM field containing a primitive $p$th root of unity $\zeta_p$. \Cref{changing-top-field} shows that we may restrict our attention to such extensions when studying the equivariant $p$-adic Artin conjecture (\cref{integrality-Phi}). Indeed, if $F/K$ is an arbitrary finite extension of totally real number fields, then let $F_\infty\colonequals FK_\infty$ be the cyclotomic $\ZZ_p$-extension, and let $\G'\colonequals\Gal(F_\infty/K)$. Then $L\colonequals F(\zeta_p)$ is a CM field with maximal totally real subfield $L^+=F(\zeta_p+\zeta_p^{-1})$, and thus $\G'$ is a quotient of $\G=\Gal(L^+_\infty/K)$.
\end{remark}

\section{Dimension reduction for integral matrices over skew power series rings}\label{sec:dimension-reduction}
Let $ K$ be a local field with valuation ring $\OO_ K$, and fix a uniformiser $\pi_ K$. Let $ K/ k$ be a cyclic Galois $p$-extension of local fields, and let $\tau\in\Gal( K/ k)$ be a generator. Let $ D$ be a finite dimensional skew field with centre $ K$, and let $\OO_ D$ denote the unique maximal $\OO_ K$-order in $ D$. Let $s/ r$ be the Hasse invariant of $ D$. Then $ D$ is obtained by adjoining an element $\pi_ D$ to $ K(\omega)$ where $\omega$ is a primitive root of unity of order $ q^s-1$ with $q$ being the order of the residue field of $ K$, moreover $\pi_ D^s=\pi_ K$ and $\pi_ D\omega\pi_ D^{-1}=\omega^{q^r}$.

Suppose that the index $s$ of $ D$ divides $\#\overline{ k}-1$ where $\#\overline{ k}$ is the order of the residue field of $ k$.
Then by \cite[Proposition~2.7]{W}, $\tau$ admits a unique extension to an automorphism of $ D$ of order $( K: k)$.

\begin{proposition}
The skew power series ring $\Sigmachi \colonequals \OO_{ D}[[X;\tau,\tau-\id]]$ is well defined, and has the following properties:
\begin{propositionlist}
	\item The total ring of quotients $\Dchi\colonequals\Quot(\OO_{ D}[[X;\tau,\tau-\id]])$ is a skew field.
	\item \label{item:Sigmachi-centre} The ring $\Sigmachi$ resp. its skew field of fractions $\Dchi$ have centres 
	\begin{align*}
		\cent(\Sigmachi)&= \OO_{ k}[[(1+X)^{( K: k)}-1]] \equalscolon \OO_{\cent(\Dchi)}, \\
		\cent(\Dchi)&=\Frac(\OO_{ k}[[(1+X)^{( K: k)}-1]]),
	\end{align*}
	where $\OO_ k$ denotes the ring of integers of $ k$.
	\item \label{Sigmachi-maximal} The ring $\Sigmachi$ is the unique maximal $\OO_{\cent(\Dchi)}$-order in $\Dchi$.
	\item \label{item:Phi} There is an embedding 
	\begin{align*}\Phi:\Dchi&\to M_{( K: k) s}\left(\Frac\left(\OO_{ K(\omega)}[[(1+X)^{( K: k)}-1]]\right)\right) \\
	\intertext{of $\Frac(\OO_ K[[(1+X)^{( K: k)}-1]])$-algebras, sending an element $\alpha\in  D$ resp. $1+X$ to}
	\alpha&\mapsto \begin{psmallmatrix}
		\phi(\alpha)\\
		&\tau(\phi(\alpha))\\
		&&\ddots\\
		&&&\tau^{( K: k)-1}(\phi(\alpha))
	\end{psmallmatrix} \\
	1+X &\mapsto \begin{psmallmatrix}
		& \mathbf{1}_{s} \\
		&& \ddots \\
		&&& \mathbf{1}_{s} \\
		(1+X)^{( K: k)}\mathbf{1}_{s}
	\end{psmallmatrix}
	\end{align*}
	where $\phi$ is the map from \eqref{eq:phi-def}, and $\mathbf{1}_{s}$ denotes the $s\times s$ identity matrix.
\end{propositionlist}
\end{proposition}
We refer to \cite[\S3]{W} for the proof {of these statements}.

In this section, we study matrices over the skew field $\Dchi$. The main result is \cref{2.13}, which allows one to do dimension reduction: by this we mean that an invertible matrix with entries in a maximal order can be replaced by a smaller matrix with the same Dieudonné determinant. This is a generalisation of a result of Nichifor and Palvannan, see \cite[Proposition~2.13]{NichiforPalvannan}, and our proof indeed follows the outline of theirs.

\begin{proposition}[Dimension reduction] \label{2.13}
	Let $m\ge1$, $n\ge1$, and
	\[A \in M_m\left( M_{n} \left( \Sigmachi \right) \right) \cap \GL_m \left( M_{n} \left( \Dchi \right) \right).\]
	Then there exists some $C\in M_n\left(\Sigmachi\right)\cap\GL_n\left(\Dchi\right)$ such that $\det A=\det C$ where $\det$ is the Dieudonné determinant.
\end{proposition}

\begin{remark}
	An equivalent formulation of \cref{2.13} would be to state it for $n=1$ instead of for all $n\ge1$. In fact, the proof goes by showing it for $n=1$, and then {constructing} an $n\times n$ matrix by taking this as the $(1,1)$-entry, filling up the rest of the diagonal with ones, and setting all off-diagonal entries to be zero. The \namecref{2.13} is stated this way because this is the form in which it will be used in the proof of \cref{thm:integrality}.
\end{remark}

\begin{proof}[Proof of \cref{2.13}]
	We begin with the following observation:
	\begin{lemma}\label{nrA-integral}
		The reduced norm of $A$ is integral: 
		\[\nr_{M_{mn}(\Dchi)/\cent(\Dchi)} A\in \OO_{\cent(\Dchi)}.\]
	\end{lemma}
	\begin{proof}[Proof of \cref{nrA-integral}]
		Actually, even more is true: the reduced characteristic polynomial of $A$ has coefficients in $\OO_{\cent(\Dchi)}$. This follows directly by applying \cref{MO10.1} with $R \colonequals \OO_{\cent(\Dchi)}$, $ K = \cent(\Dchi)$, $A \colonequals M_{mn}(\Dchi)$, and $\Delta \colonequals M_{mn}(\Sigmachi)$.
		These satisfy the conditions of \cref{MO10.1}: indeed, \cref{Sigmachi-maximal} states that $\Sigmachi$ is a maximal $\OO_{\cent(\Dchi)}$-order in $\Dchi$, and consequently $M_{mn}(\Sigmachi)$ is a maximal $\OO_{\cent(\Dchi)}$-order in $M_{mn}(\Dchi)$ by \cref{MO8.7}.
	\end{proof}
	
	Recall that we call a ring a PID if all left ideals and all right ideals are principal.
	\begin{lemma} \label{Sigmachi-ncPID}
		$\Sigmachi\left[\pi_{ D}^{-1}\right]$ is a (noncommutative) PID.
	\end{lemma}
	\begin{proof}[Proof of \cref{Sigmachi-ncPID}]
		We will show that all left ideals of $\Sigmachi[\pi_{ D}^{-1}]$ are principal; the right version can be shown in the same way, as every statement we will use has a right counterpart. The proof uses Weierstraß theory of skew power series rings, for which we refer back to \cref{sec:skew-power-series-rings}.
		
		Let $I\subseteq \Sigmachi[\pi_{ D}^{-1}]$ be a nontrivial left ideal. Then $I\cap \Sigmachi\ne(0)$: indeed, if $0\ne x\in I$, then $\pi_{ D}^a x\in I\cap\Sigmachi$ for $a\ge0$ large enough. Therefore not all reduced orders are $\infty$; let $f\in I\cap \Sigmachi$ be an element of minimal reduced order. Choosing $f$ so that it is in $\Sigmachi$ is necessary in order to be able to apply Weierstraß preparation (\cref{WPT}). This allows us to write $f$ uniquely as \[f=\epsilon F,\] where $\epsilon\in\Sigmachi^\times$ is a unit and $F\in\Sigmachi$ is a distinguished skew polynomial. It follows that $F\in I$, and $\ordred(F)=\ordred(f)$. No such theorem is available in $\Sigmachi[\pi_{ D}^{-1}]$.
		
		Clearly $\Sigmachi[\pi_{ D}^{-1}]F\subseteq I$. We will show that the converse inclusion also holds, whence $I$ is the principal left ideal generated by $F$. For this, let $g\in I$ be arbitrary. Weierstraß division (\cref{WDT}) carries over from $\Sigmachi$ to $\Sigmachi[\pi_{ D}^{-1}]$ as follows. The Weierstraß division theorem for $\Sigmachi$ states that there is a decomposition
		\begin{align*}
			\Sigmachi &= \Sigmachi F\oplus\bigoplus_{i=0}^{\ordred(F)-1} \OO_{ D}X^i. & \\
			\intertext{Inverting $\pi_{ D}$, or in other words, tensoring with $\OO_{ D}[\pi_{ D}^{-1}]$ over $\OO_{ D}$ then yields}
			\Sigmachi\left[\pi_{ D}^{-1}\right] &= \Sigmachi\left[\pi_{ D}^{-1}\right] F\oplus\bigoplus_{i=0}^{\ordred(F)-1} \OO_{ D}\left[\pi_{ D}^{-1}\right]X^i .
		\end{align*}
        With respect to this decomposition, $g$ can be written as $g=hF+r$.
		Since $g$ and $F$ are in $I$, we have $r\in I$, but {a priori, $r$ is not necessarily in $\Sigmachi$.}
		
		We claim that $r=0$. $\uplightning$ If $r\ne0$, then since $r$ has only finitely many coefficients, there is a unique $a\in\ZZ$ such that $\pi_{ D}^a r \in \OO_{ D}[X]-\pi_{ D}\OO_{ D}[X]$. We have the following inequalities:
		\[\ordred\left(\pi_{ D}^a r\right)\le \deg\left(\pi_{ D}^a r\right)=\deg(r)<\ordred(F).\]
		But since $\pi_{ D}^a r\in I$, this contradicts the minimality condition in our choice of $f$. $\lightning$
		
		Now we have $r=0$, so $g=hF$. That is, $g\in\Sigmachi[\pi_{ D}^{-1}] F$, as claimed.
	\end{proof}
	
	\cref{Sigmachi-ncPID} allows us to utilise \cref{thm:Jacobson-reduction}, which tells us that $A$ admits a diagonal reduction via elementary operations. That is, $A=UBV$ where $B\in M_{mn}(\Sigmachi[\pi_{ D}^{-1}])$ is a diagonal matrix, and $U,V\in \GL_{mn}(\Sigmachi[\pi_{ D}^{-1}])$ are products of elementary, permutation, and scalar matrices.
	
	We want to use Weierstraß preparation again, {this time on the Dieudonné determinants of these matrices}: for this, we need to clear denominators. Let $b\in\ZZ$ be the minimal integer such that $\pi_{ D}^b B\in M_{mn}(\Sigmachi)$. Then 
	\[\det(B)=\det(\pi_{ D}^{-b}\cdot \pi_{ D}^b B)=\pi_{ D}^{-b{mn}}\det(\pi_{ D}^bB)=\pi_{ D}^{-b{mn}}\beta_B J_B,\] 
	where $\beta_B\in\Sigmachi^\times$ is a unit and $J_B\in\Sigmachi$ is a monic polynomial.
	Similarly, let $u,v\in\ZZ$ be the minimal integers such that $\pi_{ D}^u U, \pi_{ D}^v V \in M_{mn}(\Sigmachi)$. Then there exist units $\beta_U,\beta_V\in\Sigmachi^\times$ and monic polynomials $J_U,J_V\in\Sigmachi$ such that 
	\begin{equation}\label{eq:UV-WPT}
		\det(U)=\pi_{ D}^{-u{mn}}\beta_U J_U, \quad \det(V)=\pi_{ D}^{-v{mn}}\beta_V J_V.
	\end{equation} 
	We claim that $J_U=J_V=1$. For this, we go back to the fact that $U$ and $V$ are products of elementary, permutation and scalar matrices. Elementary matrices have Dieudonné determinant $1$ by definition. A permutation matrix $M$ satisfies $M^2=1$ and hence $\det(M)^2=1$; in particular, $\det(M)$ is invertible in $\Sigmachi[\pi_{ D}^{-1}]$. So the elementary and permutation factors only contribute  $\beta$-factors in \eqref{eq:UV-WPT}. Moreover, since $U$ and $V$ are invertible over $\Sigmachi[\pi_{ D}^{-1}]$, their scalar matrix parts may only contribute units in $\Sigmachi[\pi_{ D}^{-1}]$, which can be written as a power of $\pi_{ D}$ times a unit in $\Sigmachi$: that is, these are accounted for in the $\pi_{ D}$- and $\beta$-factors in \eqref{eq:UV-WPT}. We conclude that $J_U=J_V=1$, as claimed.
	
	Since the Dieudonné determinant factors through the abelian group $K_1(\Dchi)$, we have
	\begin{align} \label{eq:det-A}
		\det(A)&= \det(U)\det(B)\det(V) = \left(\beta_U\beta_B\beta_V\right)\cdot J_B\cdot \pi_{ D}^{-{wmn}}=\det(C),
	\end{align}
	where {$w\colonequals u+b+v$ and} $C=\diag(\left(\beta_U\beta_B\beta_V\right)\cdot J_B\cdot \pi_{ D}^{-{wmn}},1,\ldots,1)$ is an $n\times n$ matrix. In order to prove $C\in M_{n}(\Sigmachi)\cap \GL_{n}(\Dchi)$, it remains to show that ${w}\le0$.
	
	In $\cent(\Dchi)$, we have the following equality:
	\begin{align}
		\underbrace{\nr_{M_{mn}(\Dchi)/\cent(\Dchi)}(A)}_{\in \OO_{\cent(\Dchi)}} &= \nr_{\Dchi/\cent(\Dchi)}(\det(A)) \label{eq:nr-det-A-1} \\
		&= \underbrace{\nr_{\Dchi/\cent(\Dchi)}(\beta_U\beta_B\beta_V)}_{\in \OO_{\cent(\Dchi)}^\times} \cdot \underbrace{\nr_{\Dchi/\cent(\Dchi)}(J_B)}_{\mathrm{monic}} \cdot \nr_{\Dchi/\cent(\Dchi)}\left(\pi_{ D}^{-{wmn}}\right). \label{eq:nr-det-A-2}
	\end{align}
	The first equality comes from the relationship between reduced norms and the Dieudonné determinant: for invertible matrices, $\nr_{M_{mn}(\Dchi)/\cent(\Dchi)}=\nr_{\Dchi/\cent(\Dchi)}\circ\det$, see \eqref{eq:nr-det}. Here we use the assumption that $A$ is invertible. The second equality is \eqref{eq:det-A}.
	
	On the left hand side of \eqref{eq:nr-det-A-1}, $\nr_{M_{mn}(\Dchi)/\cent(\Dchi)}(A)$ is integral, i.e. contained in $\OO_{\cent(\Dchi)}$ by \cref{nrA-integral}. On the right hand side of \eqref{eq:nr-det-A-2}, the first term is a unit because $\beta_U\beta_B\beta_V\in\Sigmachi^\times$.
	The second term is a monic polynomial by the upcoming \cref{nrd-of-monic-is-monic} because $J_B$ is monic. 
	
	Let $\pi_ k$ be a uniformiser of $\OO_ k$. Then neither the first nor the second term is divisible by $\pi_ k$, so in order for the right hand side to be integral, the third term must also be integral. The $\pi_ k$-adic valuation of $\nr_{\Dchi/\cent(\Dchi)}(\pi_{ D}^{-{wmn}})$ is a multiple of $-{wmn}$. We conclude that ${w}\le0$. This completes the proof.
\end{proof}

\begin{lemma} \label{nrd-of-monic-is-monic}
	Let $G(X)\in \Dchi=\Quot(\OO_{ D}[[X;\tau,\tau-\id]])$ be an element which is a monic polynomial in $X$.
	Then its reduced norm $\nr_{\Dchi / \cent(\Dchi)} (G)$
	in $\cent(\Dchi)=\Frac(\OO_{ k}[[(1+X)^{( K: k)}-1]])$ is also a monic polynomial in $(1+X)^{( K: k)}$. In particular, it is also a monic polynomial in $X$.
\end{lemma}
\begin{proof}
	
	We use the embedding $\Phi$ from \cref{item:Phi}.
	By definition of the reduced norm, we have
	\[\nr_{\Dchi / \cent(\Dchi)} (G) = \det(\Phi( G)) . \]
	We compute the leading term of this determinant. For this, write $G=\sum_{i=0}^{\deg G} \alpha_i \cdot (1+X)^i$ where $\alpha_i\in  D$, $\deg G$ is the degree of $G$, and $\alpha_{\deg G}=1$ as $G$ was assumed to be monic.
	\begingroup
	\allowdisplaybreaks
	\begin{align*}
		\Phi(G) 
		&= \sum_{i=0}^{\deg G} \begin{psmallmatrix}
			\phi(\alpha_i)\\
			&\tau(\phi(\alpha_i))\\
			&&\ddots\\
			&&&\tau^{( K: k)-1}(\phi(\alpha_i))
		\end{psmallmatrix} \begin{psmallmatrix}
			& \mathbf{1}_{s} \\
			&& \ddots \\
			&&& \mathbf{1}_{s} \\
			(1+X)^{( K: k)}\mathbf{1}_{s}
		\end{psmallmatrix}^i \\[1em]
		&= \sum_{i=0}^{\deg G} \begin{psmallmatrix}
			\phi(\alpha_i)\\
			&\tau(\phi(\alpha_i))\\
			&&\ddots\\
			&&&\tau^{( K: k)-1}(\phi(\alpha_i))
		\end{psmallmatrix} \begin{psmallmatrix}
			&  &  & \mathbf{1}_{s} &  &  \\
			&  &  &    & \ddots &  \\
			&  &  &    &  & \mathbf{1}_{s} \\
			(1+X)^{( K: k)}\mathbf{1}_{s} &  &  &  \\
			& \ddots &  &  \\
			&  & (1+X)^{( K: k)}\mathbf{1}_{s} & \\
		\end{psmallmatrix} (1+X)^{( K: k)j} \\
		\intertext{Here $i=( K: k) j+\ell$ {with} $0\le \ell<( K: k)$. In the second block matrix inside the summation, the first $( K: k)-\ell$ rows have entries $\mathbf 0_{s}$ and $\mathbf 1_{s}$, and the last $\ell$ rows have entries $\mathbf 0_{s}$ and $(1+X)^{( K: k)}\mathbf 1_{s}$.}
		&= \sum_{i=0}^{\deg G} \begin{psmallmatrix}
			&  &  & \phi(\alpha_i) &  &  \\
			&  &  &    & \ddots &  \\
			&  &  &    &  & \tau^{i-1}\left(\phi(\alpha_i)\right) \\
			\tau^i\left(\phi(\alpha_i)\right)\cdot(1+X)^{( K: k)} &  &  &  \\
			& \ddots &  &  \\
			&  & \tau^{( K: k)-1}\left(\phi(\alpha_i)\right)\cdot(1+X)^{( K: k)} & \\
		\end{psmallmatrix} (1+X)^{( K: k)j} 
	\end{align*}
	\endgroup
	The determinant of $\Phi(G)$ may be computed using Laplace expansion: we find that it is a polynomial in $(1+X)^{( K: k)}$. Since $\alpha_{\deg G}=1$, we have $\phi(\alpha_{\deg G})=\mathbf 1_{s}$. Therefore there is a unique highest degree term with respect to the variable $(1+X)^{( K: k)}$ in the expansion of $\det(\Phi(G))$, namely 
	\[\left((1+X)^{( K: k)}\right)^{s\cdot ( K: k) j}\cdot \left((1+X)^{( K: k)}\right)^{s\cdot r}=\left((1+X)^{( K: k)}\right)^{s\cdot \deg G},\]
	where $\deg G=( K: k) j+r$ and $0\le r<( K: k)$. Hence the determinant is monic with respect to $(1+X)^{( K: k)}$. Since $(1+X)^{( K: k)}$ is a monic polynomial in $X$, the determinant is also a monic polynomial in $X$.
\end{proof}

\section{Known cases of the equivariant \texorpdfstring{$p$}{p}-adic Artin conjecture} \label{sec:known-results}
\subsection{Maximal orders} \label{sec:integrality}
The ring $\Q(\G)$ is semisimple artinian, and it has the following Wedderburn decomposition:
\begin{equation}\label{eq:QG-W}
	\Q(\G) \simeq \bigoplus_{\chi\in\Irr(\G)/\sim_{\QQ_p}} M_{n_\chi}(D_\chi).
\end{equation}
Here $D_\chi$ is a skew field, and $\chi$ runs through equivalence classes of irreducible characters with open kernel, where two characters $\chi\sim_{\QQ_p}\chi'$ are equivalent if $\res^\G_H\chi=\sigma(\res^\G_H\chi')$ for some $\sigma\in\Gal(\QQ_{p,\chi}/\QQ_p)$. The group ring $\QQ_p[H]$ is semisimple by Maschke's theorem, and it has the following Wedderburn decomposition:
\begin{equation} \label{eq:QH}
    \QQ_p[H]\simeq\bigoplus_{\eta\in\Irr(H)/\sim_{\QQ_p}} M_{n_\eta}(D_\eta).
\end{equation}
Here the equivalence relation $\sim_{\QQ_p}$ is given as follows: $\eta\sim_{\QQ_p}\eta'$ if $\eta={}^\sigma\eta$ for some $\sigma\in\Gal(\QQ_p(\eta)/\QQ_p)$. Each $D_\eta$ is a finite dimensional skew field over its centre $\QQ_p(\eta)$. In particular, the statements in \cref{sec:skew-fields-over-local-fields} apply to $D_\eta$. Let $\OO_{D_\eta}$ denote the unique maximal $\ZZ_p$-order in $D_\eta$. 

{The Wedderburn isomorphisms \eqref{eq:QG-W} and \eqref{eq:QH} are not unique. For the rest of this section, we fix two such isomorphisms.}
The relationship between the skew fields occurring in the Wedderburn decompositions of $\Q(\G)$ and $\QQ_p[H]$ is described as follows. For details, we refer to \cite[\S4]{W}.

\begin{theorem}
Let $\chi\in\Irr(\G)$, and let $\eta\mid\res^\G_H\chi$ be an irreducible constituent of its restriction.
\begin{theoremlist}
	\item Let
	\begin{equation*}
		v_\chi\colonequals \min\left\{0<i\le w_\chi : \exists \tau\in\Gal(\QQ_p(\eta):\QQ_p), {}^{\gamma^i}\eta={}^\tau\eta\right\}.
	\end{equation*}
	Then $v_\chi$ is independent of the choice of $\eta$, and $v_\chi\mid w_\chi$. Moreover, there is precisely one $\tau\in\Gal(\QQ_p(\eta)/\QQ_p)$ such that ${}^{\gamma^{v_\chi}}\eta={}^\tau\eta$. From now on, $\tau$ denotes this automorphism.
	\item \label{item:Gal-cyclic} The automorphism $\tau\in\Gal(\QQ_p(\eta)/\QQ_p)$ has fixed field $\QQ_{p,\chi}$. Moreover, the Galois group $\Gal(\QQ_p(\eta)/\QQ_{p,\chi})$ is cyclic of order $w_\chi/v_\chi$, and $\tau$ is a generator.
	\item The automorphism $\tau\in\Gal(\QQ_p(\eta)/\QQ_{p,\chi})$ admits a unique extension as an automorphism of $D_\eta$ of order $w_\chi/v_\chi$, which we will also denote by $\tau$.
\end{theoremlist}
\end{theorem}
\begin{theorem} \label{D-conj}
	The skew field $D_\chi$ is isomorphic to
	\begin{align}
		D_\chi &\simeq {\Quot\left(\OO_{D_\eta}[[X; \tau, \tau-\id]]\right)}, \label{eq:Dchi-desc} \\
		\intertext{and its centre is}
		\cent(D_\chi)&\simeq {\Frac\left(\OO_{\QQ_{p,\chi}}[[(1+X)^{w_\chi/v_\chi}-1]]\right)}. \label{eq:cent-Dchi-desc}
	\end{align}
\end{theorem}

{Let $\Sigma_\chi$ be the preimage of $\OO_{D_\eta}[[X; \tau, \tau-\id]]$ in $D_\chi$ under the isomorphism \eqref{eq:Dchi-desc}, and let $\OO_{\cent(D_\chi)}$ denote the preimage of $\OO_{\QQ_{p,\chi}}[[\Z]]$ under the isomorphism \eqref{eq:cent-Dchi-desc}, where we write $\Z\colonequals (1+X)^{w_\chi/v_\chi}-1$ for brevity.}
\begin{remark} \label{rem:TORC}	
	{In \cite{W}, the size $n_\chi$ of the matrix ring, as well as the index of $D_\chi$ is also determined }in terms of the corresponding data for $\QQ_p[H]$, but this is not needed to address the equivariant $p$-adic Artin conjecture.
\end{remark}

The results of \cref{sec:dimension-reduction} are applicable with {$ K\colonequals \QQ_p(\eta)$, $ k\colonequals \QQ_{p,\chi}$, and $ D\colonequals D_\eta$}. Indeed, the extension is cyclic due to \cref{item:Gal-cyclic}, and the condition on the index follows from a theorem of Witt \cite[Satz 10]{Witt1952}: this shows that the index of {$D_\eta$} divides $p-1$, which in turn divides $\#\overline{\QQ_{p,\chi}}-1$.

\begin{theorem} \label{thm:integrality}
	Let $L/K$ be a finite Galois extension of number fields such that $L$ is a CM number field containing $\zeta_p$, with $\G=\Gal(L^+_\infty/K)$. Let $S$ and $T$ be finite non-empty disjoint sets of places of $K$, with $S$ containing all places ramifying in $L^+_\infty/K$ and all infinite places.
	\begin{theoremlist}
		\item \label{item:int-mo} Let $\mathfrak M(\G)$ be the preimage of 
		\[ \bigoplus_{\chi\in\Irr(\G)/\sim_{\QQ_p}} M_{n_\chi}(\Sigma_\chi)\] 
		under the Wedderburn isomorphism \eqref{eq:QG-W}. {Then $\mathfrak M(G)$} is a maximal $\Lambda(\Gamma_0)$-order in $\Q(\G)$. 
		\item \label{item:int-zeta} Assume \cref{sEIMCu} without uniqueness, and let $\zeta_S^T\in K_1(\Q(\G))$ such that $\nr\zeta_S^T=\Phi_S^T$ and $\partial(\zeta_S^T)=[Y_S^T]$. Then $\zeta_S^T$ is in the image of the natural map $\mathfrak M(\G)\cap\Q(\G)^\times\to K_1(\Q(\G))$.
		In particular, \cref{integrality-zeta} holds with $\mathfrak M=\mathfrak M(\G)$.
	\end{theoremlist}
\end{theorem}

\begin{remark}
	This can be seen as (a special case of) an equivariant version of Greenberg's result \cite[Proposition~5]{Greenberg1983} that the main conjecture implies the $p$-adic Artin conjecture (\cref{padicArtin}).
\end{remark}
\begin{proof}
	Nichifor and Palvannan's proof in the direct product case, see \cite[Theorem~1]{NichiforPalvannan}, can be modified to work in our setup, with their Proposition~2.13 replaced by our \cref{2.13}. We now explain this in more detail.
	
	First we prove \ref{item:int-mo}. According to \cref{Sigmachi-maximal}, for each $\chi$ we have that {$\OO_{D_\eta}[[X; \tau, \tau-\id]]$ is a maximal $\OO_{\QQ_{p,\chi}}[[T]]$-order in the skew field $\Quot(\OO_{D_\eta}[[X; \tau, \tau-\id]])$}. \Cref{D-conj} then shows that $\Sigma_\chi$ is a maximal $\OO_{\cent(D_\chi)}$-order in $D_\chi$. Therefore for all $\chi$, the matrix ring $M_{n_\chi}(\Sigma_\chi)$ is a maximal $\OO_{\cent(D_\chi)}$-order in $M_{n_\chi}(D_\chi)$ by \cref{MO8.7}. 
	The module $\OO_{\cent(D_\chi)}$ is finitely generated over $\Lambda(\Gamma_0)$. It follows that $M_{n_\chi}(\Sigma_\chi)$ is also a maximal $\Lambda(\Gamma_0)$-order in $M_{n_\chi}(D_\chi)$.
	Finally, $\mathfrak M(\G)$ is a maximal order in $\Q(\G)$ by \cref{MO10.5ii}, which establishes \ref{item:int-mo}.
	
	For the proof of \ref{item:int-zeta}, recall from \eqref{eq:YST-resolution} that there is a free resolution $\Lambda(\G)^m \xhookrightarrow{\alpha} \Lambda(\G)^m\twoheadrightarrow Y_S^T$ where 
	\[\alpha\in M_m(\Lambda(\G))\cap\GL_m(\Q(\G)).\] 
	In particular, $\alpha$ defines a class $[\alpha]\in K_1(\Q(\G))$. \Cref{sEIMCu}, the validity of which has been assumed, provides the following equality in $K_1(\Q(\G))$:
	\[\partial\left(\zeta_S^T\right)=\partial([\alpha]).\]
	Recall the localisation exact sequence of $K$-theory from \eqref{Witte}:
	\[K_1(\Lambda(\G))\to K_1(\Q(\G))\xrightarrow{\partial}K_0(\Lambda(\G),\Q(\G)).\]
	Using this exact sequence, there exist some $n\ge1$ and $\beta\in \GL_n(\Lambda(\G))$ such that the class $[\beta]\in K_1(\Lambda(\G))$ satisfies
	\[\zeta_S^T=[\alpha]\cdot[\beta].\]
	
	Recall from \cref{sec:Nrd-det} that the Dieudonné determinant is defined on the $K_1$-groups of $\Lambda(\G)$ and $\Q(\G)$. According to a result of Vaserstein, see \eqref{eq:Dieudonne-det-iso}, it is actually an isomorphism onto the abelianisation of the units of $\Lambda(\G)$ resp. $\Q(\G)$:
	\begin{align*}
		\det: K_1(\Lambda(\G))&\xrightarrow{\sim} \left(\Lambda(\G)^\times\right)^\ab , \\
		\det: K_1(\Q     (\G))&\xrightarrow{\sim} \left(\Q     (\G)^\times\right)^\ab .
	\end{align*} 
	Since $[\beta]\in K_1(\Lambda(\G))$, it follows that $\det \beta$ has a representative in $\Lambda(\G)$. This, together with the two isomorphisms, reduces showing integrality of $\zeta_S^T$ to showing integrality of $\det\alpha$. That is, it remains to show that $\det(\alpha)$ is in the image of the map
	\[\mathfrak M(\G)\cap \Q(\G)^\times\to \left(\Q (\G)^\times\right)^\ab.\]
	This can be done Wedderburn component-wise: that is, we need to show that for every $mn_\chi\times mn_\chi$ matrix
	\[\alpha_\chi\in M_m(M_{n_\chi}(\Sigma_\chi))\cap \GL_m(M_{n_\chi}(D_\chi))\]
	there exists an $n_\chi\times n_\chi$ matrix
	\[A_\chi\in M_{n_\chi}(\Sigma_\chi)\cap \GL_{n_\chi}(D_\chi)\]
	such that their Dieudonné determinants agree:
	\[\det \alpha_\chi=\det A_\chi.\]
	This is the statement of \cref{2.13}. The last assertion in \ref{item:int-zeta} follows from \cref{lem:integrality-equivalence}.\ref{item:integrality-equivalence-1}.
\end{proof}

\begin{corollary} \label{cor:integrality}
	Retain the setting of \cref{thm:integrality}.
	Then every maximal $\Lambda(\Gamma_0)$-order in $\Q(\G)$ is of the form
	\[\mathfrak M_{\mathbf u}(\G)\colonequals \bigoplus_{\chi\in\Irr(\G)/\sim_{\QQ_p}} u_\chi M_{n_\chi} \left( \Sigma_\chi \right) u_\chi^{-1},\]
	where $\mathbf u=(u_\chi)_{\chi}$ with $u_\chi\in \GL_{n_\chi}(D_\chi)$. Moreover, assuming \cref{sEIMCu}, $\zeta_S^T$ is in the image of the natural map $\mathfrak M_{\mathbf u}(\G)\cap\Q(\G)^\times\to K_1(\Q(\G))$.
	In particular, \cref{integrality-zeta} holds for every maximal $\Lambda(\Gamma_0)$-order in $\Q(\G)$, assuming \cref{sEIMCu}. 
\end{corollary}
\begin{proof}
	The first assertion is a consequence of \cref{Sigmachi-maximal} and basic properties of maximal orders, see {\cite[Corollary~4.10]{W}}.
	
	\Cref{item:int-zeta} states that there is a $\mathbf z=(z_\chi)_{\chi}\in\mathfrak M(\G)\cap\Q(\G)^\times$ such that $[(\mathbf z)]=\zeta^T_S\in K_1(\Q(\G))$. Clearing denominators, we write $u_\chi=a^{-1}_\chi U_\chi$ where $a_\chi\in\Sigma_\chi$ and $U_\chi \in M_{n_\chi}(\Sigma_\chi)\cap\GL_{n_\chi}(D_\chi)$. Applying dimension reduction (\cref{2.13}), we find a $b_\chi\in \Sigma_\chi\cap D_\chi^\times$ such that $\det(b_\chi)=\det(U_\chi)$. It follows that in $K_1(\Q(\G))$, we have
	\begin{align*}
		\zeta^T_S=[(\mathbf z)]=[((z_\chi))_\chi]=\left[\left(\left(a_\chi^{-1} b_\chi z_\chi b^{-1}_\chi a_\chi\right)\right)_\chi\right].
	\end{align*}
	This shows the second assertion. The third one follows from \cref{lem:integrality-equivalence}.\ref{item:integrality-equivalence-1}.
\end{proof}

\begin{corollary}
	Suppose that $L^+_\infty/K$ satisfies the $\mu=0$ hypothesis (i.e. $X_S$ is finitely generated as a $\ZZ_p$-module). Then \cref{integrality-Phi} holds for every maximal $\Lambda(\Gamma_0)$-order in $\Q(\G)$.
\end{corollary}
\begin{proof}
	In the $\mu=0$ case, the main conjecture is {known to be true} due to Ritter--Weiss \cite{RW-MC} and Kakde \cite[Theorem~2.11]{Kakde}. Using \cref{EIMCu-implies-sEIMCu}, this implies that \cref{sEIMCu} holds, and therefore the conditions of \cref{cor:integrality} are met.
\end{proof}

\subsection{\texorpdfstring{$p$}{p}-abelian extensions} \label{sec:p-abelian}
In this section, $\G'$ denotes the commutator subgroup of $\G$. Since $\G/H$ is abelian, the commutator subgroup is always contained in $H$, hence it is finite.
\begin{definition}
	$\G$ is called $p$-abelian if its commutator subgroup $\G'$ has prime-to-$p$ order.
\end{definition}
In particular, $\G$ is $p$-abelian whenever $H$ has order coprime to $p$.

\begin{theorem}[{\cite[Theorem~11.1 and Corollary~12.17]{UAEIMC}}] \label{p-abelian-EIMC}
	The following hold:
	\begin{enumerate}[label=(\roman*),ref=\roman*]
		\item \label{item:p-abelian-EIMC-i} $\G$ is $p$-abelian if and only if $\Lambda(\G)\simeq\bigoplus_{j} M_{n_j}(R_j)$, where the $R_j$ are complete commutative local rings.
		\item \label{item:p-abelian-EIMC-ii} If $\G$ has an abelian $p$-Sylow subgroup, then \cref{EIMCu} holds.
	\end{enumerate}
\end{theorem}
\begin{remark} \label{rem:p-abelian}
	If $\G$ is $p$-abelian, then $\G$ has an abelian $p$-Sylow subgroup. Indeed, a $p$-Sylow subgroup of a profinite group is---by definition---a closed pro-$p$ subgroup whose index is coprime to $p$, and by Sylow's theorem, these are all conjugate. If $\G$ is $p$-abelian, then for every $p$-Sylow subgroup $P\subseteq\G$ we have $\G'\cap P=1$; since $P'\le\G'\cap P$, this means that $P$ is abelian.  
	
	The converse does not hold.
	If some $p$-Sylow subgroup $P\subseteq \G$ is abelian, then all $p$-Sylow subgroups are abelian, because they are conjugate to $P$ by Sylow's theorem. We still have $P'\le\G'\cap P$, but this need not be an equality.
	As an explicit counterexample, consider $\G\simeq \mathfrak S_p \times \Gamma$, where $\mathfrak S_p$ is the symmetric group of degree $p$. The symmetric group has order $p!$, hence all $p$-Sylows of $\mathfrak S_p$ are of order $p$, so in particular abelian. Every $p$-Sylow of $\G$ is a direct product of a $p$-Sylow of $\mathfrak S_p$ with $\Gamma$, hence abelian. Since $p\ne2$, the commutator subgroup is the alternating group $\G'\simeq\mathfrak S_p'\simeq A_p$ (see \cite[Kapitel~II, Satz~5.1]{Huppert}), whose order $p!/2$ is divisible by $p$.
\end{remark}

In particular, \cref{p-abelian-EIMC}.\ref{item:p-abelian-EIMC-ii} combined with \cref{EIMCu-implies-sEIMCu} shows that the smoothed main conjecture (\cref{sEIMCu}) holds whenever $\G$ has an abelian $p$-Sylow subgroup, so in particular also for $p$-abelian extensions. As a direct consequence of \cref{cor:integrality}, we obtain the following:

\begin{corollary}
	Suppose that $L^+_\infty/K$ has an abelian $p$-Sylow subgroup. Then \cref{integrality-Phi} holds for every maximal $\Lambda(\Gamma_0)$-order in $\Q(\G)$. \qed
\end{corollary}
In the $p$-abelian case, \cref{p-abelian-EIMC}.\ref{item:p-abelian-EIMC-i} allows for an even stronger integrality statement.
\begin{corollary} \label{p-abelian-integrality}
	If $\G$ is $p$-abelian, then $\Phi_S^T\in\cent(\Lambda( \G))$. Moreover, \cref{integrality-Phi} holds with $\mathfrak M=\Lambda( \G)$.
\end{corollary}
\begin{proof}
	As in the proof of \cref{thm:integrality}, using \cref{sEIMCu} (which holds for $p$-abelian extensions {by \cref{p-abelian-EIMC}}), we have
	\begin{align}
		\zeta_S^T &= [\alpha]\cdot[\beta], \label{eq:pab-zeta} \\
		\intertext{where $\alpha\in M_m(\Lambda(\G))\cap \GL_m(\Q(\G))$ and $\beta\in \GL_n(\Lambda(\G))$. Replacing $\alpha$ by $\alpha\beta$, and potentially enlarging $m$ in case $n>m$, applying the reduced norm map on $K_1(\Q(\G))$ to \eqref{eq:pab-zeta} yields}
		\Phi_S^T=\nr \zeta_S^T &= \nr [\alpha] = \nr_{M_n(\Q(\G))/\cent(\Q(\G))} \alpha, \label{eq:pab-Phi}
	\end{align}
	where the first equality is again due to \cref{sEIMCu}.
	
	Let us now study the reduced norm map on $M_n(\Lambda(\G))$ in general. This is the restriction of the reduced norm map $M_n(\Q(\G))\to \cent(\Q(\G))$. Recall that $\Q(\G)$ is semisimple while $\Lambda(\G)$ is not, therefore a priori, the reduced norm on $M_n(\Lambda(\G))$ takes values in $\cent(\Q(\G))$.
	However, using the description of $\Lambda(\G)$ from \cref{p-abelian-EIMC}, we can see that in the $p$-abelian case, the reduced norm actually maps {$\Lambda(\G)$} to $\cent(\Lambda(\G))$. Indeed, the following diagram commutes (without the dashed arrow), and thus the left vertical arrow induces the dashed arrow.
	\[
	\begin{tikzcd}
		M_m(\Lambda(\G)) \arrow[rr, "\sim"] \arrow[rd, "\nr", dashed] \arrow[d,"\nr"] && \displaystyle\bigoplus_j M_{mn_j}(R_j) \arrow[d, "\bigoplus_j\det"] \\
		\cent(\Q(\G)) & \cent(\Lambda(\G)) \ar[l,hook'] \arrow[r, "\sim"] & \displaystyle\bigoplus_j R_j 
	\end{tikzcd}
	\]
	We conclude that the reduced norm of $\alpha$ on the right hand side of $\eqref{eq:pab-Phi}$ is in $\cent(\Lambda(\G))$, and hence the same is true for $\Phi_S^T$. The first assertion is now proven.
	
	Let $z\in\cent(\Lambda(\G))$ be an arbitrary central element, and let $m\colonequals 1$ in the diagram above. {Let $(z_j)_j\in \bigoplus_j R_j$ be the image of $z$ under the right-pointing bottom horizontal map.} Let $\tilde Z$ be the tuple in $\bigoplus_j M_{n_j}(R_j)$ whose $j$th component is $\diag(z_j,1,\ldots,1)$. (Note that this element is not central unless  $n_j=1$ for all indices $j$.) 
	Then the image of $\tilde Z$ under $\bigoplus_j\det$ is $(z_j)_j$.
	Let $\tilde z\in \Lambda(\G)$ be the element corresponding to $\tilde Z$ under the top horizontal isomorphism. So in the diagram above, we have
	\[\begin{tikzcd}
		\tilde z \ar[r,mapsto] \ar[d,mapsto,dashed,"\nr"] & \tilde Z \ar[d,mapsto,"\bigoplus_j\det"] \\
		z \ar[r,mapsto] & (z_j)_j
	\end{tikzcd}
	\]
	In particular, we have $\nr(\tilde z)=z$. Setting $z\colonequals\Phi_S^T$, it follows that $\Phi_S^T$ is in the image of
	\[\Lambda(\G)\cap\Q(\G)^\times\to K_1(\Q(\G))\xrightarrow{\nr}\cent(\Q(\G))^\times.\]
	This finishes the proof.
\end{proof}

\begin{remark}
	The ``if'' part of \cref{p-abelian-EIMC} shows that when $\G$ is not $p$-abelian, the reduced norm won't just be a product of determinants, and the argument above does not generalise.
\end{remark}

\printbibliography
\end{document}